\documentclass[lettersize,journal]{IEEEtran}
\usepackage{amsmath,amsfonts,amssymb}
\usepackage{mathtools}
\usepackage{algorithmic}
\usepackage{subcaption}
\usepackage{algorithm}
\usepackage{array}
\usepackage{textcomp}
\usepackage{stfloats}
\usepackage{url}
\usepackage{enumitem}
\usepackage{verbatim}
\usepackage{graphicx}
\usepackage[dvipsnames]{xcolor}
\PassOptionsToPackage{hyphens}{url}\usepackage{hyperref}
\hypersetup{
    colorlinks,
    linkcolor={red!50!black},
    citecolor={blue!50!black},
    urlcolor={blue!80!black}
}

\usepackage[backend=bibtex,style=ieee]{biblatex}
\usepackage{CoOL-bib/customCommands}

\newcommand{\diag}{\mathrm{diag}}
\renewcommand{\mod}{\mathrm{\:mod\:}}

\usepackage{amsthm}

\theoremstyle{plain}
\newtheorem{theorem}{Theorem}
\newtheorem{lemma}{Lemma}
\newtheorem{proposition}{Proposition}
\newtheorem{corollary}{Corollary}

\theoremstyle{definition}
\newtheorem{remark}{Remark}
\newtheorem{example}{Example}

\newcounter{constnum}
\newcommand{\const}[1]{\refstepcounter{constnum}\label{#1}}

\newcounter{hpnum}
\newcommand{\hp}[1]{\refstepcounter{hpnum}\label{#1}}

\begin{document}

\title{Non-Asymptotic Pointwise and Worst-Case Bounds for Classical Spectrum Estimators}

\author{Andrew Lamperski
  \thanks{This work was supported in part by NSF CMMI-2122856}
  \thanks{Electrical and Computer Engineering, University of Minnesota, \texttt{alampers@umn.edu}}
  }



\maketitle

\begin{abstract}
  Spectrum estimation is a fundamental methodology in the analysis of time-series data, with applications including medicine, speech analysis, and control design. The asymptotic theory of spectrum estimation is well-understood, but the theory is limited when the number of samples is fixed and finite. This paper gives non-asymptotic error bounds for a broad class of spectral estimators, both pointwise (at specific frequencies) and in the worst case over all frequencies. The general method is used to derive error bounds for the classical Blackman-Tukey, Bartlett, and Welch estimators. In particular, these are first non-asymptotic error bounds for Bartlett and Welch estimators.  
\end{abstract}


\section{Introduction}
\IEEEPARstart{S}{p}ectrum estimation is the problem of estimating the power spectral density of a random signal from a finite collection of samples of a time-series. Its applications include analysis of heart and neural signals, identification of dynamic systems for control, and speech analysis \cite{stoica2005spectral}.

The asymptotic theory of spectrum estimation is well-understood \cite{stoica2005spectral,liu2010asymptotics}. Here, the behavior of the power spectral density estimate is characterized as the amount of data tends to infinity. Additionally, when the estimates are assumed to be Gaussian, the bias and variance of the estimates are known. 

In contrast, the non-asymptotic theory of spectral estimation is quite limited. The non-asymptotic theory aims to characterize the error of spectral estimates when the number of samples is fixed and finite. Existing works on non-asymptotic spectral analysis are \cite{fiecas2019spectral}, which analyzes smoothed periodogram estimates (not covered by this paper), and \cite{zhang2021convergence,veedu2021topology} which examine variants of the Blackman-Tukey estimator (similar to Theorem~\ref{thm:BTconvergence} of this paper). Other closely-related works are \cite{doddi2022efficient}, which gives a non-asymtotic analysis of regularized Weiner filters, \cite{chang2022statistical}, which derives central limit theorem-type results for the estimator class from \cite{zhang2021convergence}, and \cite{krampe2022frequency}, which builds a variety of hypothesis tests from the estimator class from \cite{zhang2021convergence}. 

Over the last decade, the theory of non-asymptotic statistical estimation has reached a substantial level of maturity, with good introductory texts given by \cite{wainwright2019high,vershynin2018high}. However, most work focuses on independent data. For time-series, non-trivial dependencies exist between the samples, precluding many of the techniques used for independent data. In the related area of dynamic system identification, \cite{hardt2018gradient,lee2020non,oymak2019non,tsiamis2019finite,sarkar2019near}, specialized methods have been developed to bound identification errors from dependent data.

The main contribution of this paper is a framework for deriving non-asymptotic error bounds for a broad class of spectrum estimators. These bounds hold pointwise in frequency and in the worst-case across all frequencies.  
We derive specific error bounds for Blackman-Tukey, Bartlett, and Welch estimators. In order to get explicit constants for all error bounds, we derive explicit constants in the classical Hanson-Wright inequality, which may be of independent interest.

The paper is arranged as follows. The problem and class of estimators are described in Section~\ref{sec:problem}. Section~\ref{sec:gen} gives the general framework for non-asymptotic error analysis and the errors of classical estimators are bounded in Section~\ref{sec:classical}. Conclusions are given in Section~\ref{sec:conclusion}. All proofs are in the appendices.

\paragraph*{Notation}
Random variables are denoted in bold, e.g. $\bx$. $\bbE[\bx]$ is the expected value of $\bx$, $\bbP(\bcE)$ is the probability of event $\bcE$. If $\bx$ is a scalar-valued random variable and $p\ge 1$, then $\|\bx\|_p = \left(\bbE\left[|\bx|^p\right]\right)^{1/p}$. 
If $M$ is a matrix, then $M^\top$ is the transpose, $M^\star$ is the conjugate transpose, and $\overline{M}$ is the complex conjuage. For a vector, $x$, and $p\in[1,\infty]$, $\|x\|_p$ is the $\ell_p$ norm, while for a matrix, $M$, $\|M\|_2$ denotes the induced $2$-norm (i.e. the maximum singular value), and $\|M\|_F$ denotes the Frobenius norm. $A\otimes B$ is the Kroneckter product of matrices $A$ and $B$. $1_{m\times n}$ and $0_{m\times n}$ are  the $m\times n$ matrices of ones and zeros, respectively. $I_n$ is the  $n\times n$ identity matrix.. $\bbN$ is the set of non-negative integers, $\bbZ$ is the set of  integers, $\bbR$ is the set of real numbers, and $\bbC$ is the set of complex numbers. $\diag(x)$ is the square matrix formed by placing the entries of a vector $x$ on the diagonal. The trace of a square matrix, $M$, is denoted by $\Tr(M)$. The ceiling function is denoted by $\ceil*{\cdot}$. The modulo operation between two numbers is denoted by $x \mod y$. In other words, if $x=ky+r$ for $k\in\bbZ$ and $r\in [0,y)$, then $x\mod y = r$.

\section{Problem Setup}
\label{sec:problem}

Let $\by[k]$ be a stationary zero-mean $\bbR^n$-valued discrete-time stochastic process with respective autocovariance sequence and power spectral density give by:
\begin{align*}
  R[k] &= \bbE\left[\by[i+k]\by[i]^\top\right] \\
  \Phi(s) &= \sum_{k=-\infty}^{\infty} e^{-j2\pi sk} R[k]
\end{align*}

We assume that one of the following conditions holds:
\begin{enumerate}[label=A\arabic*), series=assumption]
\item
  \label{a:gaussian}
  $\by[k]$ is Gaussian
\item
  \label{a:subgaussian}
  There is an impulse  response sequence $h[k]\in\bbR^{n\times m}$ such that
  $\by[k] = \sum_{\ell=-\infty}^{\infty} h[k-\ell]\bzeta[\ell]$, where $\bzeta[k]=\begin{bmatrix}\bzeta_1[k] & \cdots & \bzeta_m[k]\end{bmatrix}^\top$ such that for $i=1,\ldots,m$ and for $k\in \bbZ$, $\bzeta_i[k]$ are independent $\sigma$-sub-Gaussian random variables.
\end{enumerate}

By $\sigma$-sub-Gaussian, we mean that $\bbE\left[e^{\lambda \bzeta_i[k]}\right]\le e^{\frac{\sigma^2\lambda^2}{2}}$ for all $\lambda\in\bbR$. Inequality (\ref{eq:subgaussianVar}) from Lemma~\ref{lem:subgaussianConstants} in Appendix~\ref{app:constants} implies that  $\sigma\ge 1$. 

In the case of Assumption~\ref{a:subgaussian}, we will have
\begin{equation}
  \label{eq:ftToSpec}
  \Phi(s) = H(s)H(-s)^\top = H(s)H(s)^\star,
\end{equation}
where $H$ is the discrete-time Fourier transform of $h$.

Let $\hat \bPhi(s)$ be an estimate of $\Phi(s)$ constructed from samples $\by[0],\ldots,\by[N-1]$. 
The main goals of this paper are to derive high-probability bounds on pointwise estimation error:
$$
\|\Phi(s)-\hat\bPhi(s)\|_2,
$$
for all $s\in  \left[-\frac{1}{2},\frac{1}{2}\right]$ and worst-case estimation error: 
$$
\sup_{s\in\left[-\frac{1}{2},\frac{1}{2}\right]}\|\Phi(s)-\hat\bPhi(s)\|_2.
$$

In both cases, the first step of the analysis is to bound the pointwise estimation error:
\begin{equation}
\label{eq:triangle}
\|\Phi(s)-\hat\bPhi(s)\|_2\le \left\|\Phi(s)\hspace{-1pt}-\hspace{-1pt}\bbE\left[\hat\bPhi(s)\right]\right\|_2
\hspace{-2pt}+
\left\|\hat \bPhi(s)\hspace{-1pt}-\hspace{-1pt}\bbE\left[\hat\bPhi(s)\right]\right\|_2,
\end{equation}
for all $s\in \left[-\frac{1}{2},\frac{1}{2}\right]$. 

The first term on the right of (\ref{eq:triangle}) corresponds to the bias of the estimate, while the second corresponds to the concentration of the estimate around its expected value.

To get concrete bounds on the bias and concentration terms, we need to explicitly fix the class of estimators considered. Let $\bY = \begin{bmatrix}\by[0] & \by[1] & \cdots & \by[N-1]\end{bmatrix} \in\bbR^{n\times N}$. We focus on estimators of the form 
\begin{equation}
  \label{eq:quadraticPeriodogram}
\hat\bPhi(s) = \bY D(-s)A D(s) \bY^\top
\end{equation}
where $D(s)=\mathrm{diag}\left(\begin{bmatrix}1 & e^{j2\pi s} &  \cdots & e^{j2\pi (N-1)s} \end{bmatrix}\right)$ and $A\in\bbR^{N\times N}$ is a symmetric matrix. 

\section{General Results}
\label{sec:gen}

This section gives a collection of error bounds on the class of estimators defined by (\ref{eq:quadraticPeriodogram}). In particular, we bound the pointwise concentration of $\hat\bPhi(s)$ to its mean, the worst-case concentration of $\hat\bPhi(s)$ to its mean, and the bias of the estimator. The pointwise concentration bounds can be expressed in terms of $A$. The worst-case and bias bounds require different quantities which can be derived from $A$. 



To prove worst-case bounds, it is helpful to re-write
(\ref{eq:quadraticPeriodogram}) as  
\begin{equation}
  \label{eq:expandedEstimator}
\hat\bPhi(s) = \sum_{k=-N+1}^{N-1}e^{-j2\pi sk} \bY B[k] \bY^\top 
\end{equation}
where $B[k]$ is defined by:
\begin{subequations}
  \label{eq:expandedMatrices}
\begin{align}
  d[k] &= \begin{cases}
        \begin{bmatrix}
          A_{k,0} & \cdots & A_{N-1,N-1-k}
        \end{bmatrix}^\top & k\ge 0 \\
         \begin{bmatrix}
          A_{0,|k|} & \cdots & A_{N-1-|k|,N-1}
        \end{bmatrix}^\top
        & k < 0       
    \end{cases} \\
  B[k] &= \begin{cases}
\setlength\arraycolsep{2pt}
    \begin{bmatrix}
      0_{k\times (N-k)} & 0_{k\times k} \\
      \diag(d[k]) & 0_{(N-k)\times k}
    \end{bmatrix} & k\ge 0 \\
\setlength\arraycolsep{2pt}
    \begin{bmatrix}
      0_{(N-|k|)\times |k|} & \diag(d[k]) \\
       0_{|k|\times |k|} & 0_{|k|\times (N-|k|)}
    \end{bmatrix} & k<0.
  \end{cases}
\end{align}
\end{subequations}

In the analysis, we will utilize:
\begin{subequations}
  \label{eq:expandedNorms}
\begin{align}
  \|B[k]\|_2 &= \|d[k]\|_\infty \\
  \|B[k]\|_F &= \|d[k]\|_2.
\end{align}
\end{subequations}

Now we describe the bias. 
The expected value of spectral estimators of the form (\ref{eq:quadraticPeriodogram}) can be expressed as
$$
\bbE\left[\hat\bPhi(s) \right] = \sum_{k=-N+1}^{N-1} e^{-j2\pi sk} b[k] R[k],
$$
where
\begin{equation}
  \label{eq:biasGeneral}
b[k] = \begin{cases}
  \sum_{i=k}^{N-1}A_{i,i-k} & 0\le k < N \\
  \sum_{i=|k|}^{N-1}A_{i+k,i} & -N < k < 0 \\
  0 & |k|\ge N.
\end{cases}
\end{equation}

Note that for $|k|<N$, $b[k]$ can be expressed equivalently as $b[k] = 1_{1\times (N-|k|)} d[k]$.

Now the bias can be expressed as:
\begin{subequations}
  \begin{align}
    \label{eq:biasCombined}
  \MoveEqLeft[0]
\Phi(s)-
  \bbE\left[\hat\bPhi(s) \right] = \sum_{k=-\infty}^{\infty}e^{-j2\pi sk}(1-b[k])R[k] \\
    \label{eq:biasSplit}
 & \hspace{-10pt} =\sum_{k=-N+1}^{N-1}e^{-j2 \pi sk}(1-b[k])R[k] +  \sum_{|\ell| \ge N} e^{-j2\pi s\ell} R[\ell]. 
\end{align}
\end{subequations}

From (\ref{eq:biasSplit}), we see that a small bias can only be obtained when $R[k]$ decays appropriately as $|k|\to \infty$. To this end, let
$$
\|R\|_1 = \sum_{k=-\infty}^{\infty} \|R[k]\|_2.
$$

We assume that $\|R\|_1 < \infty$. This is a typical assumption for the convergence of discrete-time Fourier transforms and holds in many common classes of processes.  For example, when $\Phi(s)=H(s)H(s)^\star$ where $H$ is a stable rational transfer matrix, we have that $\|R[k]\|_2 \le \gamma \rho^{|k|}$ for some constants $\gamma>0$ and $\rho \in [0,1)$. However, the assumption would fail in the case of bandlimited spectra such as
$$
\Phi(s) = \begin{cases} 1 & |s| \le W < \frac{1}{2}\\
  0 & |s| > W
  \end{cases}.
$$

Now we describe some specialized notation used to present our general results on the error of spectral estimators of the form (\ref{eq:quadraticPeriodogram}).

\const{ConcentrationMult}
\const{ConcentrationExp}
\const{ConcentrationSubGauss}

Define constants $c_{\ref{ConcentrationMult}}$, $c_{\ref{ConcentrationExp}}$, and $c_{\ref{ConcentrationSubGauss}}$ by
\begin{subequations}
\label{eq:constants}
  \begin{align}
    \textrm{Assumption}~\ref{a:gaussian} &\implies c_{\ref{ConcentrationMult}}=2,
                                           \quad  c_{\ref{ConcentrationExp}}=\frac{1}{32},
                                           &c_{\ref{ConcentrationSubGauss}}=1 \\
    \textrm{Assumption}~\ref{a:subgaussian} &\implies
                                              c_{\ref{ConcentrationMult}}=4,
                                              \quad c_{\ref{ConcentrationExp}}=2^{-19},
                                              &c_{\ref{ConcentrationSubGauss}}=\sigma.
  \end{align}
  \end{subequations}

Let $\|\Phi\|_{\infty}=\sup_{s\in\left[-\frac{1}{2},\frac{1}{2}\right]}\|\Phi(s)\|_2$. We assume that $\|\Phi\|_{\infty} < \infty$. 

For $\epsilon > 0 $ and $\delta \in (0,1)$  the following quantities will be used in the error bounds below:
\hp{pointwiseHP}
\hp{worstCaseGridHP}
\begin{subequations}
  \label{eq:collectedConstants}
\begin{align}
  \label{eq:dataLowerBound}
  \alpha(\epsilon) & = \max\left\{\frac{c_{\ref{ConcentrationSubGauss}}^4\|\Phi\|_{\infty}^2}{\epsilon^2},\frac{c_{\ref{ConcentrationSubGauss}}^2\|\Phi\|_\infty}{\epsilon}\right\} \\
  \beta(\delta) &= \frac{\log\left(\delta^{-1} 10^{2n}  c_{\ref{ConcentrationMult}}\right)}{c_{\ref{ConcentrationExp}}}\\
   \label{eq:minWindow}
  \hat M(\epsilon) &= \inf\left\{\tilde  M\in \bbN \middle| \sum_{|k|\ge \tilde M} \|R[k]\|_2 \le \frac{\epsilon}{2}\right\}.
\end{align}
\end{subequations}

Note that when $\|R[k]\|_2\le \gamma \rho^{|k|}$ for all $k$, we can bound $\hat M(\epsilon)\le \max\left\{0,\frac{\log\left(\frac{(1-\rho)\epsilon}{2\gamma} \right)}{\log \rho}\right\}$.

The following theorem gives sufficient conditions for achieving low estimation error with high probability. It is proved in Appendix~\ref{app:genConvergence}.

\begin{theorem}
  \label{thm:genConvergence}
   Define $\alpha$, $\beta$,  and $\hat M$ as in (\ref{eq:collectedConstants}). For all $\epsilon >0$ and all $\delta \in (0,1)$,
  \begin{enumerate}
  \item \label{it:concentration}
    If $\frac{1}{\max\{\|A\|_2,\|A\|_F^2\}}\ge \alpha(\epsilon)\beta(\delta)$, then for all $s\in \left[-\frac{1}{2},\frac{1}{2} \right] $ we have
    $$ \bbP\left(\left\|\hat\bPhi(s)-\bbE\left[\hat\bPhi(s)\right] \right\|_2 >\epsilon \right) \le \delta.$$
  \item \label{it:worst}
    Let $g\ge \max\{\|A\|_2,\|A\|_F^2\}$ and, $g\ge \max\{\|B[k]\|_2,\|B[k]\|_F^2\}$ for all $|k| < N$. Assume that there is a number $\hat N\le N$ such that  $B[k]=0$ for $|k|\ge \hat N$. If $\frac{1}{g}\ge \alpha(\epsilon/2)\left(\log(5\hat N^2) + \beta(\delta/2)\right)$ 
    then 
    $$
     \bbP\left(\sup_{s\in\left[-\frac{1}{2},\frac{1}{2}\right]}\left\|\hat\bPhi(s)-\bbE\left[\hat\bPhi(s)\right] \right\|_2 >\epsilon \right) \le \delta.
     $$
   \item \label{it:bias}
     Assume that $b[k]\in [0,1]$ for all $k\in \bbZ$. If $b[k]\ge 1-\frac{\epsilon}{2\|R\|_1}$ for $|k|< \hat M(\epsilon)$, then
     $$
     \sup_{s\in \left[-\frac{1}{2},\frac{1}{2}\right]} \left\|\Phi(s)-\bbE\left[\hat\bPhi(s) \right]\right\|_2 \le \epsilon.
     $$
   \item \label{it:fullPointwise} If the conditions of both \ref{it:concentration}) and \ref{it:bias}) are satisfied, then for all $s\in \left[-\frac{1}{2},\frac{1}{2}\right]$ we have
     $$
     \bbP\left(\left\|\hat\bPhi(s)-\Phi(s) \right\|_2 >2\epsilon \right) \le \delta.$$
   \item \label{it:fullWorst} If the conditions of both \ref{it:worst}) and \ref{it:bias}) are satisfied then
      $$\bbP\left(\sup_{s\in\left[-\frac{1}{2},\frac{1}{2}\right]}\left\|\hat\bPhi(s)-\Phi(s) \right\|_2 >2\epsilon \right) \le \delta.$$
  \end{enumerate}
\end{theorem}

The following corollary gives alternative ways of expressing the error bounds from Theorem~\ref{thm:genConvergence}. It is proved in Appendix~\ref{app:corGen}
\begin{corollary}
  \label{cor:gen}
  \begin{enumerate}
  \item \label{it:pointwiseCor}
    Let $\xi(A)=\max\{\|A\|_2,\|A\|_F^2\}$. For all $s\in \left[-\frac{1}{2},\frac{1}{2}\right]$, and all $\delta\in (0,1)$, the following holds with probability at least $1-\delta$:
    \begin{multline*}
      \left\|\hat\bPhi(s)-\bbE\left[\hat\bPhi(s)\right]\right\|_2
      \\
      \le c_{\ref{ConcentrationSubGauss}}^2\|\Phi\|_{\infty}\max\{\xi(A)\beta(\delta),\sqrt{\xi(A)\beta(\delta)}\}
    \end{multline*}
  \item
\label{it:worstCor}
    Let $g\ge \max\{\|A\|_2,\|A\|_F^2\}$ and, $g\ge \max\{\|B[k]\|_2,\|B[k]\|_F^2\}$ for all $|k| < N$. Assume that there is a number $\hat N\le N$ such that  $B[k]=0$ for $|k|\ge \hat N$. Set $\hat\beta(\hat N,\hat \delta)= \log(5\hat N^2)+\beta\left(\frac{\delta}{2}\right)$. Then for all $\delta\in (0,1)$, the following bound holds with probability at least $1-\delta$:
     \begin{multline*}
      \left\|\hat\bPhi-\bbE\left[\hat\bPhi\right]\right\|_{\infty}\le
      \\
       2 c_{\ref{ConcentrationSubGauss}}^2\|\Phi\|_{\infty}\max\left\{g\hat\beta(\hat N,\delta),\sqrt{g\hat\beta(\hat N,\delta)}\right\}
     \end{multline*}
   \item \label{it:biasDecay}
     Assume that there are constants $\gamma >0$ and $\rho \in [0,1)$ such that $\|R[k]\|_2 \le \gamma \rho^{|k|}$ for all $k\in\bbZ$ and assume that $b[k]=0$ for all $|k|\ge \hat N$, where $\hat N\le N$. Then
     $$
     \left\|\Phi-\bbE\left[\hat\bPhi\right]\right\|_{\infty}\le \gamma \sum_{k=-\hat N+1}^{\hat N-1}|1-b[k]|\rho^{|k|}+\frac{2\gamma \rho^{\hat N}}{1-\rho}
     $$
   \item
\label{it:unknownNorm}
     If $\left\|\Phi-\bbE\left[\hat\bPhi\right]\right\|_{\infty} \le b$, and $a:= 2 c_{\ref{ConcentrationSubGauss}}^2\max\left\{g\hat\beta(\hat N,\delta),\sqrt{g\hat\beta(\hat N,\delta)}\right\}<1$, then with probability at least $1-\delta$
     $$
     \left\|\hat\bPhi-\Phi \right\|_{\infty}\le \frac{a\|\hat\bPhi\|_{\infty}+b}{1-a}
     $$
  \end{enumerate}
\end{corollary}

\begin{remark}
  \label{rem:tradeoffs}
  In the Blackman-Tukey, Bartlett, and Welch algorithms discussed below, the number $\hat N\le N$ is a tunable parameter that can be used to specify a trade-off between bias and variance. In each of these algorithms, we will have $g=O(\hat N/N)$, so the probabilistic error bound from part~\ref{it:worstCor}) scales as $ O\left(\sqrt{\frac{\hat N \log \hat N}{N}}\right)$ in each of these cases. In particular, the bound from part~\ref{it:worstCor}) increases monotonically with $\hat N$, while the bound from part~\ref{it:biasDecay}) typically decreases monotonically with $\hat N$. In the next section, we will give explicit bounds for the Bartlett estimator, and show how to optimize over $\hat N$ to give a total error bound of $\|\hat\bPhi-\Phi\|_{\infty}=\tilde O(N^{-1/3})$, ignoring logarithmic factors. Similar bounds are likely possible for Blackman-Tukey and Welch estimators, but these will depend on the specific window functions used for these methods.  
\end{remark}

\begin{remark}
  \label{rem:bias}
  To use the bounds from Corollary~\ref{cor:gen} in practice, we need some assumptions about the decay of the autocovariance, we can bound the bias, as in part~\ref{it:biasDecay}). (See the next paragraph for more details.) These assumptions could be obtained from domain knowledge, such as time constant estimates or prior noise characterizations. Then, part~\ref{it:unknownNorm}) can be used to derive bounds on the total worst-case error just from the bound on the bias, $b$, the estimated spectrum, $\hat\bPhi$, and the term $a$, which scales like $\tilde O\left(\sqrt{\hat N/N}\right)$. As discussed in Remark~\ref{rem:tradeoffs}, the truncation parameter, $\hat N$, can typically be tuned to optimize the resulting bound. 

Unfortunately, it is not possible to estimate the autocovariance decay parameters, $\gamma$ and $\rho$, without some assumptions.
Indeed, consider the pathological autocovariance sequence
$$
R[k]=\begin{cases}
  2 & k=0 \\
  1 & k=\pm D \\
  0 & k\notin \{-D,0,D\},
\end{cases}
$$
which could be obtained by running white noise through the filter with impulse response $h[k]=\delta[k]+\delta[k-D]$, where $\delta[\cdot]$ is the Kronecker delta. This signal 
would be indistinguishable from white noise when the data set has size $N<D$, and so the decay constants from part~\ref{it:biasDecay}) would artificially appear to be $\gamma=2$ and $\rho=0$. In reality, the constants would need to satisfy $\gamma \rho^D \ge 1$.
\end{remark}

\begin{remark}
  \label{rem:asymptotic}

  In numerical experiments in Section~\ref{sec:numerical}, we see that the bounds for Gaussian variables are rather conservative (1-2 orders of magnitude greater than true error), while the bounds for sub-Gaussian variables are highly conservative (5-8 orders of magnitude than true error). Decreasing the gap between Gaussian and sub-Gaussian bounds would require improving the constants in the Hanson-Wright inequality, which is outside of the scope of this paper. 

  In contrast, the bounds obtained from asymptotic analysis are comparatively tight, often on the same order of magnitude of the true error. See, e.g. Section 5.7 of \cite{brillinger1981time}.
  While these asymptotic bounds are  less conservate, they rely on unquantified approximations. Specifically, they utilize asymptotic distributions without quantifying the error induced by approximating the distribution with its asymptotic distribution.

  The existing asymptotic results indicate that more precise, frequency-dependent bounds that depend on fewer assumptions should be obtainable. For scalar signals, the asymptotic variance scales with $\Phi(s)^2$ for  smoothed periodograms \cite{brillinger1981time} and the Blackman-Tukey method \cite{liu2010asymptotics}. The bounds in \cite{brillinger1981time}, for example, just rely on bounds of various moments and cumulants, rather than assumptions of Gaussian or sub-Gaussian distributions. 
  In contrast, the non-asymptotic bounds from part~\ref{it:concentration}) of Theorem~\ref{thm:genConvergence} and part~\ref{it:pointwiseCor}) of Corollary~\ref{cor:gen} are the same across frequency. The asymptotic results indicate it may be possible to obtain more precise error bounds that depend on the specific value of $\Phi(s)$ at frequency $s$. Furthermore, it may be possible to relax the Gaussian/sub-Gaussian assumptions, though this would require a fundamentally different proof approach.
\end{remark}

\section{Error Bounds for Specific Classical Spectrum Estimators}
\label{sec:classical}

This section shows how to analyze periodograms, Blackman-Tukey estimators, Bartlett estimators, and Welch estimators in terms of the general result from \ref{thm:genConvergence}. In particular, high probability error bounds are obtained in the case of Blackman-Tukey, Bartlett, and Welch estimators. For periodograms, the bias is bounded, but high-probability bounds cannot be obtained, consistent with classical calculations on variance of periodograms. (See \cite{stoica2005spectral}.)

The definitions of the various estimators follows the presentation from \cite{stoica2005spectral}, and it is shown how each estimator can be expressed in the form of (\ref{eq:quadraticPeriodogram}). This leads to a unified approach to error analysis. All of the propositions and theorems of this section are proved in Appendix~\ref{app:specific}.

\subsection{Periodograms}

The standard biased autocovariance sequence estimate is defined by
\begin{equation}
  \label{eq:biasedACS}
\hat \bR[k] = \begin{cases}\frac{1}{N}\sum_{i=k}^{N-1}\by[i]\by[i-k]^\top & 0\le k < N \\
    \frac{1}{N}\sum_{i=-k}^{N-1} \by[i+k] \by[i]^\top & -N < k < 0 \\
    0 & |k| \ge N
  \end{cases}
\end{equation}
The corresponding periodogram is given by
$$
\hat \bPhi(s)=\sum_{k=-N+1}^{N+1}e^{-j2\pi sk} \hat \bR[k].
$$
In this case, $\hat \bPhi(s)$ can be expressed in the form of (\ref{eq:quadraticPeriodogram}) with $A = \frac{1}{N} 1_{N\times N}$, the scaled matrix of ones. Here we have $\|A\|_2=\|A\|_F=1$. As a result, the conditions of Theorem~\ref{thm:genConvergence} Part \ref{it:concentration}) on pointwise error cannot be met for $\xi\ge 1$.  Similarly, the conditions of Part~\ref{it:worst}) cannot be met.
So, the most we can bound using Theorem~\ref{thm:genConvergence} is the bias:
\begin{proposition}
  \label{prop:biasedPeriodogram}
  Let $\hat M(\epsilon)$ be defined in (\ref{eq:collectedConstants}).
  If $N\ge \frac{2 \hat M(\epsilon) \|R\|_1}{\epsilon}$, then
  $$
  \sup_{s\in \left[-\frac{1}{2},\frac{1}{2}\right]} \left\|\Phi(s)-\bbE\left[\hat\bPhi(s) \right]\right\|_2 \le \epsilon.
  $$
\end{proposition}



The unbiased autocovariance sequence estimate is given by:
\begin{equation*}
 \tilde \bR[k] = \begin{cases}\frac{1}{N-|k|}\sum_{i=k}^{N-1}\by[i]\by[i-k]^\top & 0\le k < N \\
    \frac{1}{N-|k|}\sum_{i=-k}^{N_1} \by[i+k] \by[i]^\top & -N < k < 0 \\
    0 & |k| \ge N
  \end{cases} 
\end{equation*}
The unbiased\footnote{The autocovarience  sequence estimate is unbiased in this case. However, the periodogram itself is biased since we are not measuring correlations more than $N$ steps apart} periodogram estimate is 
$$
\hat\bPhi(s) = \sum_{-k=-N+1}^{N-1} e^{-j2\pi s k} \tilde \bR[k] = \bY D(-s)AD(s)\bY^\top,
$$
where $A$  is a Toeplitz matrix given by:
$$
A = \begin{bmatrix}
  \frac{1}{N} & \frac{1}{N-1} & \cdots & \frac{1}{1} \\
  \frac{1}{N-1}& \frac{1}{N} & \cdots & \frac{1}{2} \\
  \vdots & \vdots & & \vdots \\
  \frac{1}{1} & \frac{1}{2} & \cdots & \frac{1}{N}
  \end{bmatrix}
$$
In this unbiased case,
  $$
  1 \le \left(\frac{1}{\sqrt{N}}1_{N\times 1}\right)^\top A 
\left(\frac{1}{\sqrt{N}}1_{N\times 1}\right) \le \|A\|_2\le \|A\|_F,
  $$
  for all values of $N$.
As a result, the conditions of Theorem~\ref{thm:genConvergence} Part \ref{it:concentration}) on pointwise error cannot be met for $\xi\ge 1$.  Similarly, the conditions of Part~\ref{it:worst}) cannot be met.
 Again, all we can bound is the bias:
  \begin{proposition}
    \label{prop:unbiasedPeriodogram}
  Let $\hat M$ be defined in (\ref{eq:collectedConstants}).
  If $N\ge \hat M(\epsilon)$, then
  $$
  \sup_{s\in \left[-\frac{1}{2},\frac{1}{2}\right]} \left\|\Phi(s)-\bbE\left[\hat\bPhi(s) \right]\right\|_2 \le \epsilon.
  $$
  \end{proposition}


  \subsection{Blackman-Tukey Estimators}
  Let $\hat \bR[k]$ be the biased autocovariance sequence estimate from (\ref{eq:biasedACS}). For $M\le N$ and a window function $w:\bbZ\to \bbR$
  define the Blackman-Tukey estimate by:
  $$
  \hat\Phi(s) = \sum_{k=-M+1}^{M-1} e^{-j2\pi s k} w[k] \hat \bR[k]
  $$
  In this case, $\hat\bPhi$ can be expressed as in (\ref{eq:quadraticPeriodogram}), where $A$ is a Toeplitz matrix defined by:
  \begin{equation}
    \label{eq:ABT}
  A = \frac{1}{N}
\setlength\arraycolsep{2pt}
  \begin{bmatrix}
    w[0] & \cdots & w[-M+1] &&& 0 \\
    \vdots & && & \ddots \\
    w[M-1] &&\ddots&&& w[-M+1] \\
    &\ddots&&&& \vdots \\
    0&&w[M-1]&&\cdots& w[0]
  \end{bmatrix}.
\end{equation}
For symmetry of $A$, we must have $w[k]=w[-k]$.

  For many common windows, such as the rectangular, Bartlett, Hann, Hamming, and Blackman windows, the entries satisfy $w[i] \in [0,1]$ for $i=-M+1,\ldots,M-1$. Under these assumptions, the theorem below gives sufficient conditions for the Blackman-Tukey method to give low error with high probability. The bounds on $\|\hat\bPhi(s)-\Phi(s)\|_2$ are omitted, as  they are direct consequences of parts \ref{it:fullPointwise}) and \ref{it:fullWorst}) of Theorem~\ref{thm:genConvergence}.

  \begin{theorem}
    \label{thm:BTconvergence}
    Define $\alpha$, $\beta$, and $\hat M$ as in (\ref{eq:collectedConstants}).
  \begin{enumerate}
  \item \label{it:BTconcentration}
    If $\frac{N}{2M-1}\ge \alpha(\epsilon)\beta(\delta)$, then for all $s\in \left[-\frac{1}{2},\frac{1}{2} \right] $ we have
    $$ \bbP\left(\left\|\hat\bPhi(s)-\bbE\left[\hat\bPhi(s)\right] \right\|_2 >\epsilon \right) \le \delta.$$
  \item \label{it:BTworst}
    If $\frac{N}{2M-1}\ge \alpha(\epsilon/2)\left(\log(5M^2)+\beta(\delta/2)\right)$ then 
    $$
     \bbP\left(\sup_{s\in\left[-\frac{1}{2},\frac{1}{2}\right]}\left\|\hat\bPhi(s)-\bbE\left[\hat\bPhi(s)\right] \right\|_2 >\epsilon \right) \le \delta.
     $$
   \item \label{it:BTbias}
If  $M\ge \hat M$, $N \ge \frac{2\hat M \|R\|_1}{\epsilon}$, $w[k]\ge \frac{1-\frac{\epsilon}{2\|R\|_1}}{1-\frac{|k|}{N}}$ for $|k|<\hat M$, and $w[k] \in [0,1]$ for $|k| \ge \hat M$, then
     $$
     \sup_{s\in \left[-\frac{1}{2},\frac{1}{2}\right]} \left\|\Phi(s)-\bbE\left[\hat\bPhi(s) \right]\right\|_2 \le \epsilon.
     $$
  \end{enumerate}
\end{theorem}

In the notation of Theorem~\ref{thm:genConvergence} and Corollary~\ref{cor:gen}, $g=\frac{2M-1}{N}$, $\hat N=M$, and $b[k]=(N-|k|)w[k]/N$. See the proof for more details. 

\begin{remark}
  A set of non-asymptotic worst-case spectral error bounds were obtained in Theorems 4.1 and 4.2 of \cite{zhang2021convergence}. These correspond to the special case of the Blackman-Tukey estimate when $w$ is defined from a kernel. These results appear a bit different from Theorem~\ref{thm:BTconvergence} since \cite{zhang2021convergence} uses different assumptions and bounds the error using a different norm.

  Another related non-asymptotic worst-case bound is achieved in Theorem 6 of \cite{veedu2021topology}. The estimator in this paper is a truncated periodogram which can be shown to be a specialized type of Blackman-Tukey estimator.  
\end{remark}

 \subsection{Bartlett Estimators}
For the Bartlett estimator, assume that $N=LM$, where $L$ and $M$ are positive integers.
The Bartlett estimator is given by:
\begin{align*}
  \hat \by_i(s) &=\sum_{k=0}^{M-1}e^{-j2\pi sk}\by[iM+k] & \textrm{for } i=0,\ldots,L-1 \\
  \hat \bPhi(s)&=\frac{1}{N}\sum_{i=0}^{L-1}\hat\by_i(s)\hat\by_i(s)^\star
\end{align*}

The Bartlett estimator can be represented in the form of (\ref{eq:quadraticPeriodogram}) where $A$ is the block diagonal matrix:
\begin{equation}
  \label{eq:ABartlett}
A = \frac{1}{N}\begin{bmatrix}
  1_{M\times M} \\
  & \ddots & \\
  && 1_{M\times M}
  \end{bmatrix}
\end{equation}
where there are $L$ blocks of size $M\times M$.

  \begin{theorem}
  \label{thm:BartlettConvergence}
  Define $\alpha$, $\beta$, and $\hat M$, as in  (\ref{eq:collectedConstants}).
  \begin{enumerate}
  \item \label{it:BartlettConcentration}
    If $\frac{N}{M}\ge \alpha(\epsilon)\beta(\delta)$, then for all $s\in \left[-\frac{1}{2},\frac{1}{2} \right] $ we have
    $$ \bbP\left(\left\|\hat\bPhi(s)-\bbE\left[\hat\bPhi(s)\right] \right\|_2 >\epsilon \right) \le \delta.$$
  \item \label{it:BartlettWorst}
    If $\frac{N}{M}\ge \alpha(\epsilon/2)\left(
    \log(5M^2)+\beta(\delta/2)\right)$ then 
    $$
     \bbP\left(\sup_{s\in\left[-\frac{1}{2},\frac{1}{2}\right]}\left\|\hat\bPhi(s)-\bbE\left[\hat\bPhi(s)\right] \right\|_2 >\epsilon \right) \le \delta.
     $$
   \item \label{it:BartlettBias}
     If  $M\ge \frac{2 \hat M(\epsilon) \|R\|_1}{\epsilon}$, then
     $$
     \sup_{s\in \left[-\frac{1}{2},\frac{1}{2}\right]} \left\|\Phi(s)-\bbE\left[\hat\bPhi(s) \right]\right\|_2 \le \epsilon.
     $$
  \end{enumerate}
\end{theorem}

In the notation of Theorem~\ref{thm:genConvergence} and Corollary~\ref{cor:gen}, $g=\frac{M}{N}$, $\hat N=M$, and $b[k]=1-|k|/M$. See the proof for more details.

In the special case that $\|R[k]\|_2\le \gamma \rho^{|k|}$ for all $k$, the bias has a more explicit bound given by:
\begin{multline*}
\sup_{s\in \left[-\frac{1}{2},\frac{1}{2}\right]} \left\|\Phi(s)-\bbE\left[\hat\bPhi(s) \right]\right\|_2\le \\ \frac{2\gamma \rho}{(1-\rho)^2 M}+
2\gamma \left(\frac{\rho^2}{(1-\rho)^2}+\frac{1}{1-\rho}\right)\rho^M = O(M^{-1}).
\end{multline*}

The bias bound can be combined with the high-probability bound from Corollary~\ref{cor:gen}, part~\ref{it:worstCor} to show that
$$
\|\hat\bPhi-\Phi\|_{\infty}=\tilde O(\sqrt{M/N}+M^{-1})
$$
with high probability, where $\tilde O$ suppresses logarithmic factors. Optimizing over $M$ leads to $M=O(N^{1/3})$, leading to an overall error bound of $O(N^{-1/3})$.

\subsection{Welch Estimators}
For the Welch estimator, assume that $N = (S-1)K+M$ for positive integers $S$, $K$, and $M$. Let $v\in\bbR^M$ be a window function. The Welch estimator is defined by:
\begin{subequations}
\begin{align}
  \label{eq:WelchFT}
  \hat \by_i(s) &=\sum_{k=0}^{M-1}e^{-j2\pi sk} \frac{v[k]}{\|v\|_2}\by[iK+k] & \textrm{for } i=0,\ldots,S-1 \\
  \hat\bPhi(s) &= \frac{1}{S}\sum_{i=0}^{S-1}\hat \by_i(s) \hat\by_i(s)^\star  
\end{align}
\end{subequations}

In this case $\bPhi(s)$ can be expressed in the form of (\ref{eq:quadraticPeriodogram}) with $A$ a sum of block-diagonal matrices:
\begin{equation}
  \label{eq:AWelch}
  A = \frac{1}{S\|v\|_2^2}\sum_{i=0}^{S-1}\begin{bmatrix}
    0_{iK\times iK}  \\
    & vv^\top \\
    && 0_{(N-iK-M)\times (N-iK-M)}
  \end{bmatrix}.
\end{equation}

  \begin{theorem}
  \label{thm:WelchConvergence}
  Define $\alpha$, $\beta$, and $\hat M$ as in (\ref{eq:collectedConstants}).
  \begin{enumerate}
  \item \label{it:WelchConcentration}
    If $\frac{S}{1+2\frac{M}{K}}\ge \alpha(\epsilon)\beta(\delta)$, then for all $s\in \left[-\frac{1}{2},\frac{1}{2} \right] $ we have
    $$ \bbP\left(\left\|\hat\bPhi(s)-\bbE\left[\hat\bPhi(s)\right] \right\|_2 >\epsilon \right) \le \delta.$$
  \item \label{it:WelchWorst}
    If $\frac{S}{1+2\frac{M}{K}}\ge \alpha(\epsilon/2)\left(\log(5M^2)+\beta(\delta/2)\right)$ then 
    $$
     \bbP\left(\sup_{s\in\left[-\frac{1}{2},\frac{1}{2}\right]}\left\|\hat\bPhi(s)-\bbE\left[\hat\bPhi(s)\right] \right\|_2 >\epsilon \right) \le \delta.
     $$
   \item \label{it:WelchBias}
     If  $M\ge\hat M(\epsilon)$ and for all $|k|<\hat M(\epsilon)$ we have $ \sum_{i=|k|}^{M-1}\frac{v[i-|k|]v[i]}{\|v\|_2^2}\ge 1-\frac{\epsilon}{2\|R\|_1}$, then
     $$
     \sup_{s\in \left[-\frac{1}{2},\frac{1}{2}\right]} \left\|\Phi(s)-\bbE\left[\hat\bPhi(s) \right]\right\|_2 \le \epsilon.
     $$
  \end{enumerate}
\end{theorem}

In the notation of Theorem~\ref{thm:genConvergence} and Corollary~\ref{cor:gen}, $g=\frac{1+2\frac{M}{K}}{S}$, $\hat N=M$, and $b[k]=(N-|k|)w[k]/N$. See the proof for more details. In typical applications,  the ratio $r=\frac{M}{K}$ is fixed with $r>1$. A commonly used value is $r=2$. In this case, $g\le \frac{M(1+2r)}{N}=O(\hat N/N)$, as in Remark~\ref{rem:tradeoffs}. 

\section{Numerical Studies}
\label{sec:numerical}
Here we show two applications of the bounds from this paper to simulated stochastic processes. In all cases, the Welch algorithm with Hann window was used. 

\begin{example}
  \label{ex:scalar}
  We consider a scalar signal of the form
  $$
  \by[k]=\frac{1}{1-\rho^2}\sum_{\ell=1}^{\infty}\rho^{\ell-1}\bzeta[k-\ell]
  $$
  where $\bzeta[k]$ are scalar-valued IID random variables with mean zero and variance $1$ and $\rho=0.3$. In this case, the corresponding autocovariance is exactly $R[k]=\rho^{|k|}$ for all $k$. We simulated the case that $\bzeta[k]$ are Gaussian and also when $\zeta[k]$ is uniform over $\left[-\sqrt{3},\sqrt{3} \right]$, which is $\sqrt{3}$-sub-Gaussian. As can be seen, the bound for the Gaussian process, is somewhat conservative, while the sub-Gaussian processes is quite conservative. The reason for the conservatism of the sub-Gaussian process is the large constant factor arising from the sub-Gaussian Hanson-Wright inequality.

  \begin{figure}
    \begin{minipage}[b]{\columnwidth}
      \centering
      \includegraphics[width=.8\textwidth]{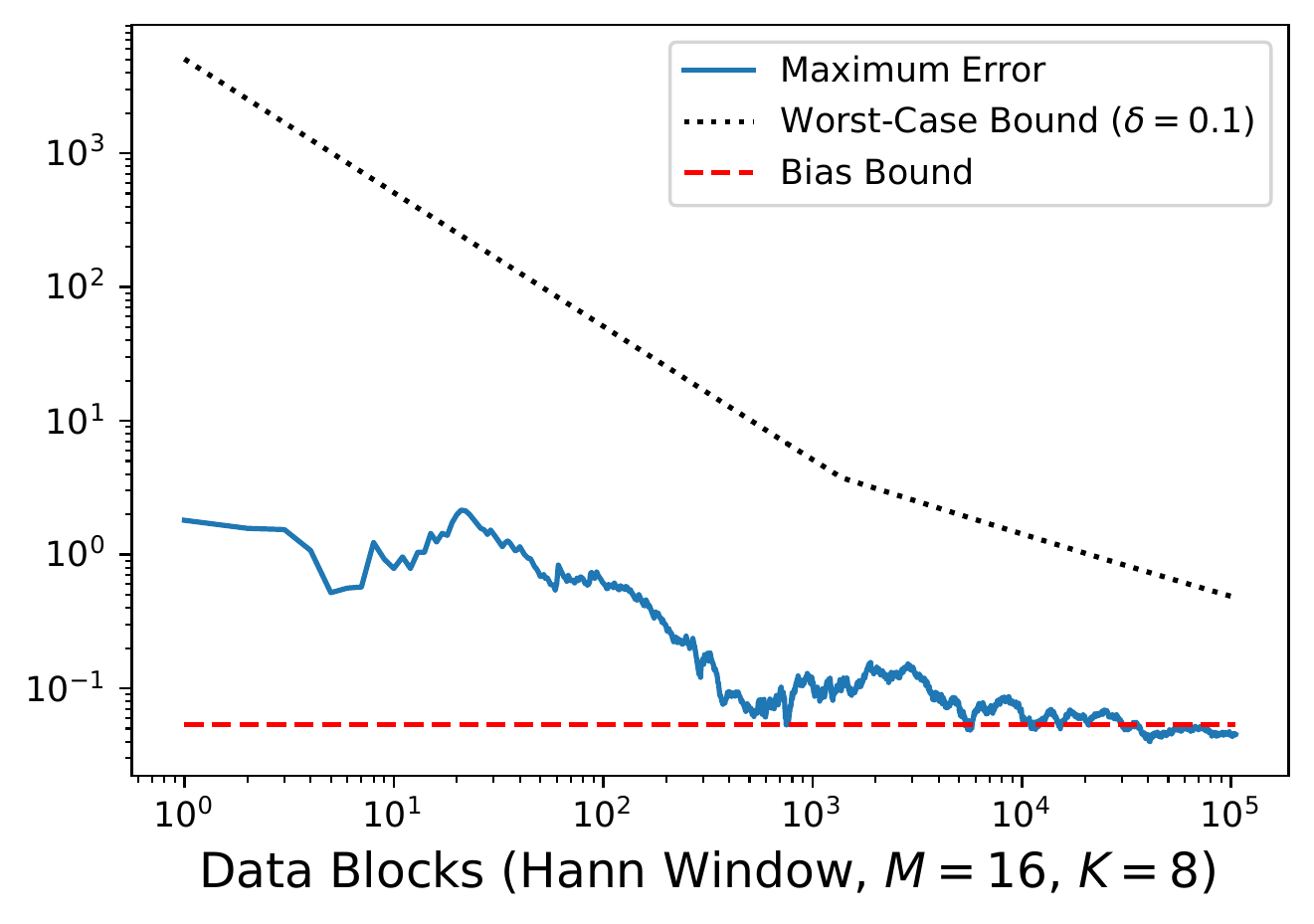}
      \subcaption{Scalar Gaussian Process}
    \end{minipage}
    \begin{minipage}[b]{\columnwidth}
      \centering
      \includegraphics[width=.8\textwidth]{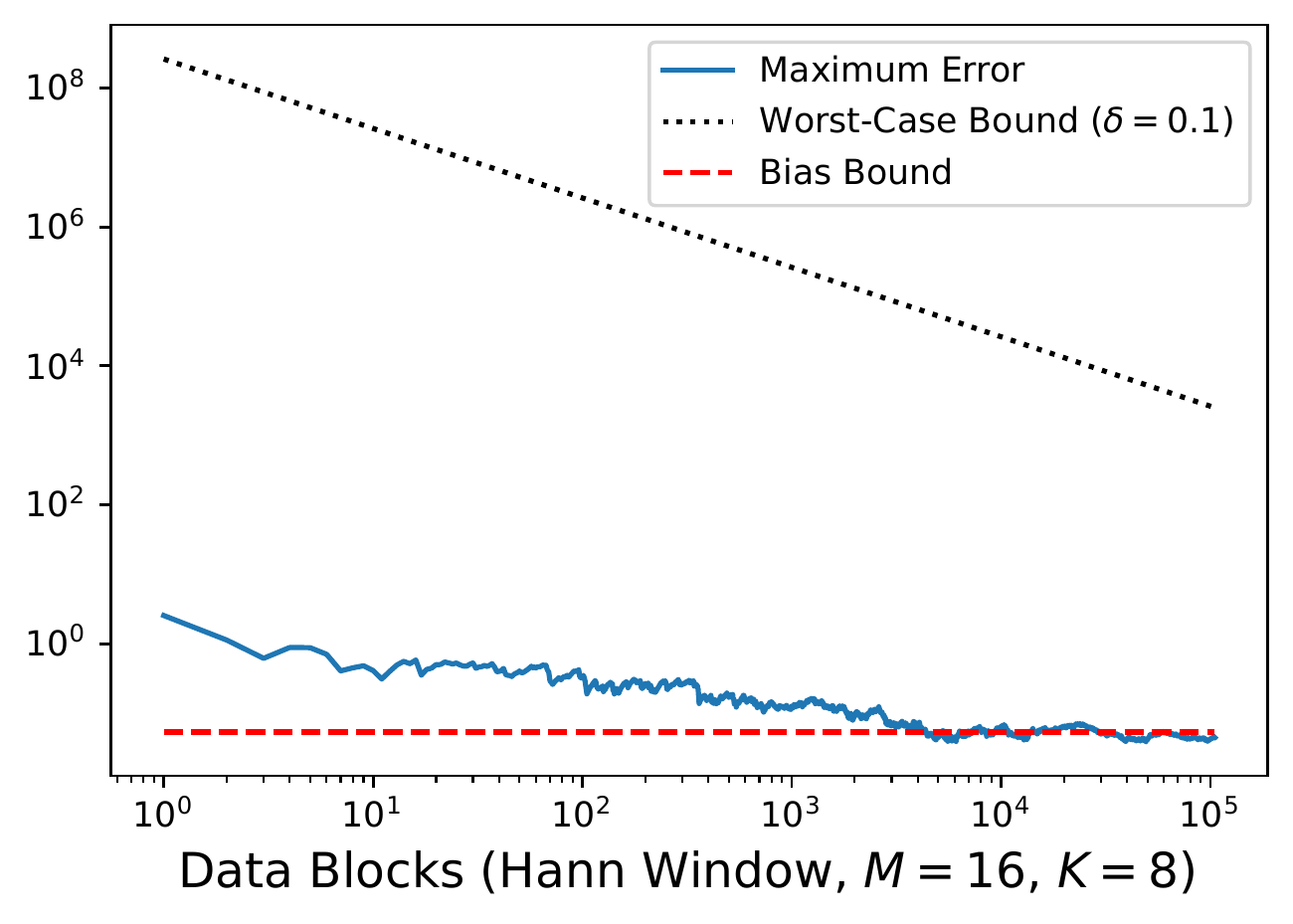}
      \subcaption{Scalar Sub-Gaussian Process}
    \end{minipage}
    \caption{\label{fig:scalar} Error of the Welch method for Example~\ref{ex:scalar}. The number of data blocks corresponds to $S$. The blue line shows the maximum error over a linearly spaced grid of $[0,.5]$ of size $101$, the black dotted line shows the total worst-case error from parts~\ref{it:worstCor}) and \ref{it:biasDecay}) Corollary~\ref{cor:gen}, and the red dashed line shows exact bias.}
  \end{figure}
\end{example}

\begin{example}
  The next example shows the results for a process of the form $y[k]=\sum_{\ell=-\infty}^\infty h[k-\ell]\zeta[\ell]$ where, $\zeta[\ell]\in\bbR^3$ are IID Gaussians with zero mean and identity covariance, 
  $$
  h[k]=\begin{cases}
    D & k=0 \\
    CA^{k-1}B & k\ge 1\\
    0 & k<0
  \end{cases}
  $$
  and
  \begin{align*}
    A &= \begin{bmatrix}
      0.3 & 0 \\
      1 & 0.3
    \end{bmatrix} & B&=\begin{bmatrix}
      1 & 0 & 0 \\
      0 & 1 & 0
    \end{bmatrix} \\
    C &=\begin{bmatrix}
      0 & 0 \\
      1 & 0 \\
      0 & 1
      \end{bmatrix} & D&=\begin{bmatrix}1 & 0 & 0 \\ 0 & 1 & 0 \\ 0 & 0 & 1\end{bmatrix}.
  \end{align*}
  In this case $\|R[k]\|_2\le \gamma \rho^{|k|}$ for any $\rho \in (0.3,1)$ and  sufficiently large $\gamma$. Specifically, if $P$ is a positive definite matrix with condition number $\kappa >0$ such that $A^\top P A\preceq \rho^2 P$, and $X=AXA^\top +BB^\top$ is the observability Gramian, then 
  \begin{multline*}
    \gamma = \max\left\{\|CXC^\top +DD^\top\|_2,\right.\\
    \left.\sqrt{\kappa}\|C\|_2\left(\frac{\|BD\|_2}{\rho}+\|XC^\top\|_2 \right)\right\}.
  \end{multline*}
  Then an upper  bound on the bias can be computed explicitly from Corollary~\ref{cor:gen}.

  As can be seen in Fig.~\ref{fig:chain}, the bounds are a bit conservative, as in the scalar case. 

  \begin{figure}
    \centering
    \includegraphics[width=.8\columnwidth]{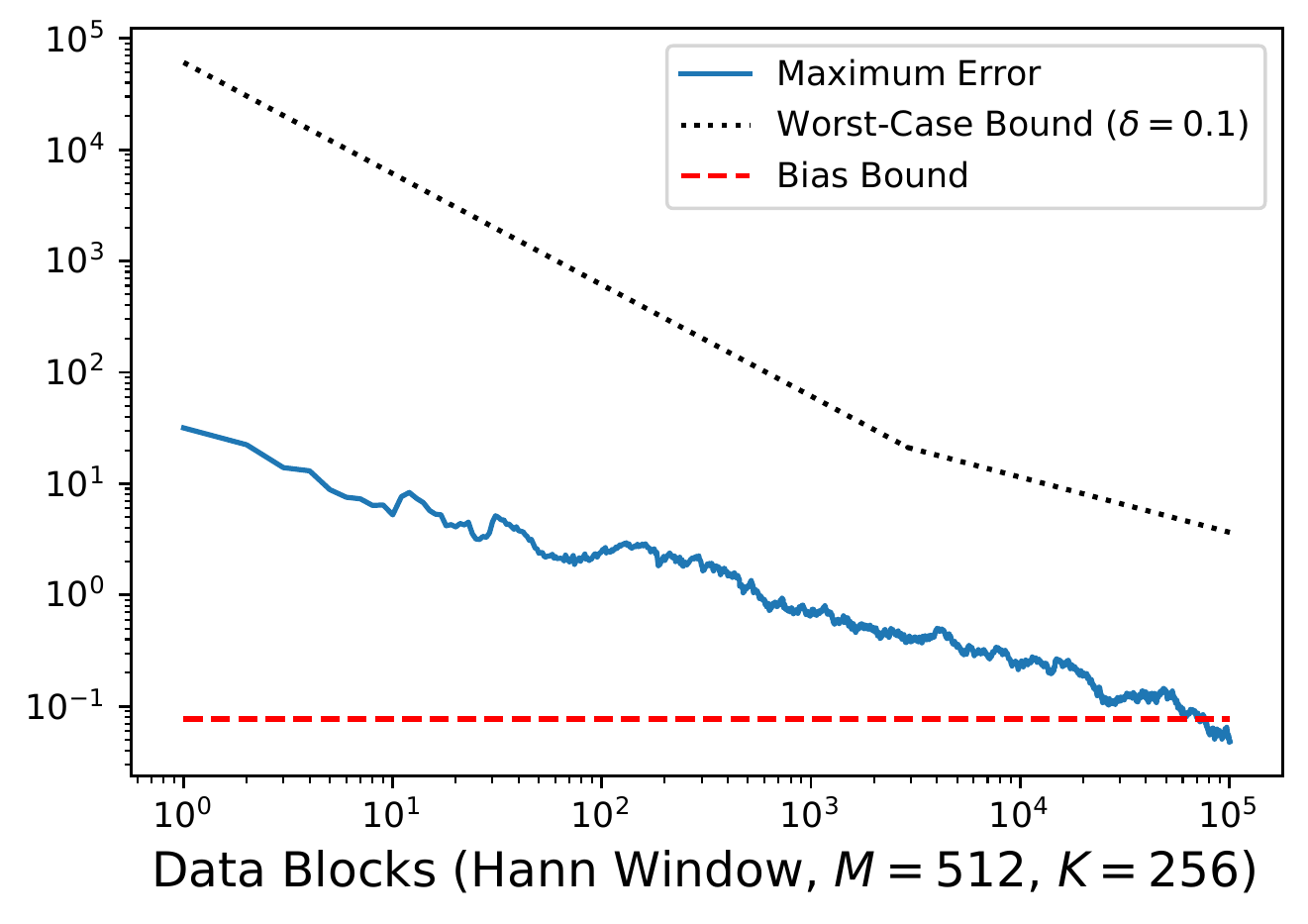}
    \caption{\label{fig:chain} The error of the $3$-dimensional signal. For details on the lines, see Fig~\ref{fig:scalar}. The only difference is now the red dashed line is an upper bound on the bias, rather than an exact bias. }
  \end{figure}
  
\end{example}

\section{Conclusion}
\label{sec:conclusion}

This paper gives a method for deriving non-asymptotic error bounds for a class of spectrum estimators. This method is used to derive error bounds for a variety of classical estimators. Many avenues for future work remain. Window-dependent bias-variance trade-offs can be formulated for the Welch and Blackman-Tukey estimators. Errors induced by preprocessing steps such as centering, normalization, and detrending could be quantified. More precise, frequency-dependent error bounds may be possible, in analogy with asymptotic results, and the Gaussian/sub-Gaussian assumptions could potentially be relaxed.  
The bounds from the paper could be utilized to bound errors in estimating $\Phi^{-1}(s)$, which is particularly useful for network identification \cite{materassi2012problem} and system identification \cite{ljung1999system}.  

\printbibliography

\appendices

\section{Concentration for Time-Series Data Matrices}
\label{app:genConcentration}

This section presents an intermediate result that is used to prove the probabilistic bounds in Theorem~\ref{thm:genConvergence}.

\begin{lemma}
  \label{lem:matrixConcentration}
  Let $J\in\bbC^{N\times N}$. Assume that either $J\in\bbR^{N\times N}$ or $J$ is Hermitian. Let $\bY = \begin{bmatrix}\by[0] & \cdots & \by[N-1]\end{bmatrix}\in\bbR^{n\times N}$ be a matrix of data satisfying either Assumption~\ref{a:gaussian} or Assumption~\ref{a:subgaussian}. 
For all $\epsilon >0$
\begin{multline*}
\bbP\left(\left\|\bY J \bY^\top - \bbE\left[\bY J \bY^\top \right]\right\|_2 > \epsilon \right) \le \\
10^{2n} c_{\ref{ConcentrationMult}}\exp\left(-c_{\ref{ConcentrationExp}} \min\left\{\frac{\epsilon^2}{c_{\ref{ConcentrationSubGauss}}^4\|J\|_F^2\|\Phi\|_{\infty}^2},\frac{\epsilon}{c_{\ref{ConcentrationSubGauss}}^2\|J\|_2\|\Phi\|_\infty}\right\}\right),
\end{multline*}
where 
$c_{\ref{ConcentrationMult}}$, $c_{\ref{ConcentrationExp}}$, and $c_{\ref{ConcentrationSubGauss}}$ are defined in (\ref{eq:constants}). 
\end{lemma}

To prove Lemma~\ref{lem:matrixConcentration}, we first derive concentration results for the scalar random variables $u^\star \bY J \bY^\top v$, with $\|u\|_2 =\|v\|_2=1$. These bounds are obtained by decoupling the dependent data and then using the
Hanson-Wright inequality. Some specialized results for the case of Gaussian data are utilized to achieve tighter constant factors.  

\subsection{Preliminary Results for the Scalarized Problem}

Let $u,v\in\bbC^n$ be such that $\|u\|_2=1$, $\|v\|_2=1$, and let $\underline{\by}=\begin{bmatrix}\by[0]^\top & \cdots & \by[N-1]^\top \end{bmatrix}^\top$ be the vertical stack of the data.

\begin{lemma}
  The scalarized random variable, $u^\star \bY J \bY^\top v$ satisfies
  \begin{align*}
    u^\star \bY J \bY^\top v = \underline{\by}^\top \left(J^\top \otimes (vu^\star) \right) \underline{\by}
  \end{align*}
  where $\|J^\top \otimes (vu^\star)\|_2=\|J\|_2$ and $\|J^\top \otimes (vu^\star)\|_F = \|J\|_F$. 
\end{lemma}

\begin{IEEEproof}
  The alternate formula for the variable follows from direct calculation:
  \begin{align*}
    u^\star \bY J \bY^\star v &= \sum_{p,q=0}^{N-1} (u^\star \by[p]) J_{p,q} (\by[q]^\top v) \\
                              &= \sum_{p,q=0}^{N-1}\by[q]^\top \left( J_{p,q} vu^\star\right) \by[p] \\
                              &=\underline{\by}^\top \left(J^\top \otimes (vu^\star) \right) \underline{\by}
  \end{align*}
  The norm properties follow from direct calculation as well:
  \begin{equation*}
    \|J^\top \otimes (vu^\star) \|_2 =\|J^\top \|_2 \|vu^\star \|_2 = \|J\|_2
  \end{equation*}
  and
  \begin{align*}
    \|J^\top \otimes (vu^\star)\|_F^2
    &=\Tr\left(
      \overline{J} J^\top \otimes uv^\star v u^\star
      \right) \\
    &= \Tr\left(
      (\overline{J}\otimes u)(J^\top \otimes u^\star)
      \right) \\
    &= \Tr\left(
      (J^\top \overline{J})\otimes (1)
      \right) \\
    &=\|J\|_F^2. 
  \end{align*}
\end{IEEEproof}

Let
$$
\underline{R} = \bbE\left[
  \underline{\by} \underline{\by}^\top
\right] = \begin{bmatrix}
  R[0] & R[-1] &  \cdots & R[-N+1] \\
  R[1] & R[0] & \cdots & R[-N+2] \\
  \vdots & \vdots & &\vdots \\
  R[N-1] & R[N-2] & \cdots & R[0]. 
\end{bmatrix}
$$

The matrix $\underline{R}$ will be utilized to express the correlated data vectors in terms of contributions of independent random variables. The following bound will be utilized to analyze the concentration of these decoupled vectors. 

\begin{lemma}
  \label{lem:covFromSpec}
  The matrix $\underline{R}$ satisfies $\|\underline{R}\|_2\le \|\Phi\|_{\infty}$. 
\end{lemma}

\begin{IEEEproof}
Since $\underline{R}$ is real-valued, symmetric, and positive semidefinite 
$$
\| \underline{R} \|_2 = \sup_{\|z\|_2= 1}z^\top \underline{R}z 
$$
where the supremum ranges over complex-valued unit vectors. 

Let $z = \begin{bmatrix}z[0]^\star & \cdots & z[N-1]^\star \end{bmatrix}^\star\in\bbC^{nN}$ be a unit vector with $z[k]\in\bbC^n$. Identify $z$ with a discrete-time signal by setting $z[k]=0$ for $k<0$ and $k\ge N$. Let $\hat z(s)$ be the Fourier transform of the signal, $z$. Then convolution rule and Plancharel theorem imply:
\begin{align*}
  z^\star \underline{R}z &= \sum_{k,\ell=-\infty}^{\infty}z[k]^\star R[k-\ell]z[\ell] \\
  &=\int_{-\frac{1}{2}}^{\frac{1}{2}} \hat z(s)^\star \Phi(s) \hat z(s)ds \\
  &\le \|\Phi\|_{\infty}
\end{align*}
Thus, $\|\underline{R}\|_2\le \|\Phi\|_{\infty}$. 
\end{IEEEproof}

\subsection{Special Results for the Gaussian Case}

The following lemma is a specialized version of the Hanson-Wright inequality for Gaussian random variables. See Exercise 2.17 of \cite{wainwright2019high}. 
\begin{lemma}
  \label{lem:gaussianHS}
  Let $A\in\bbC^{n\times n}$. Assume that either $A\in\bbR^{n\times n}$ or $A$ is Hermitian. 
  If $\bx$ is a Gaussian random vector with mean $0_{n\times 1}$ and covariance $I_n$, then for all $\epsilon\ge 0$:
  \begin{multline*}
  \bbP\left(\bx^\top A \bx - \bbE\left[\bx^\top A \bx \right]  > \epsilon \right)\le \\ \exp\left(-\frac{1}{8}\min\left\{\frac{\epsilon^2}{\|A\|_F^2},\frac{\epsilon}{\|A\|_2} \right\}\right).
  \end{multline*}
\end{lemma}
\begin{IEEEproof}
  Let $B=\frac{1}{2}(A+A^\top)$. Then under either assumption about $A$, $B$ is a real symmetric matrix such that $\bx^\top A \bx=\bx^\top B\bx$, $\|B\|_2\le \|A\|_2$, and $\|B\|_F\le \|A\|_F$.
  
Let $V$ be an orthogonal matrix such that $B=V\diag(\lambda)V^\top$, where $\lambda = \begin{bmatrix}\lambda_1 & \cdots & \lambda_n\end{bmatrix}^\top$ are the eigenvalues of $B$. Let $\by =  V^\top \bx$ so that
$$
\bx^\top A\bx = \bx^\top B \bx = \sum_{i=1}^n \lambda_i \by_i^2.
$$
Now $\by_i$ are independent Gaussian random variables with mean $0$ and variance $1$. 

Since $\|B\|_F=\|\lambda\|_2$ and $\|B\|_2=\|\lambda\|_{\infty}$, it follows that $\bx^\top A\bx$ is $(2\|B\|_F,4\|B\|_2)$-sub-exponential. Due to the inequalities, it must also be $(2\|A\|_F,4\|A\|_2)$-sub-exponential. The result then follows from Proposition 2.9 of \cite{wainwright2019high}.
\end{IEEEproof}

\begin{lemma}
  Let Assumption~\ref{a:gaussian} hold, so that $\by$ is a zero-mean Gaussian process. Let $J\in \bbC^{N\times N}$, $u\in\bbC^{n}$, $v\in \bbC^n$ be unit vectors such that one of the following conditions holds:
  \begin{enumerate}
  \item $J\in \bbR^{N\times N}$, $u\in \bbR^n$, and $v\in\bbR^n$ or
  \item $J$ is  Hermitian and $u=v$. 
  \end{enumerate}
  Then, for any $\epsilon >0$ the following bound holds: 
  \begin{multline*}
    \bbP\left(u^\star \bY J \bY^\top v - \bbE\left[ u^\star \bY J \bY^\top v\right] > \epsilon \right) \le \\
    \exp\left(-\frac{1}{8}\min\left\{
            \frac{\epsilon^2}{\|J\|_F^2 \|\Phi\|_\infty^2},
            \frac{\epsilon}{\|J\|_2 \|\Phi\|_{\infty}}
          \right\}\right).
  \end{multline*}
\end{lemma}

\begin{proof}
  If $\by$ is a Gaussian process then $\underline{\by}$ is identically distributed to $G\bx$ where $\bx$ is a Gaussian random vector with mean $0$ and covariance $I$ and $GG^{\top} = \underline{R}$. So then $u^\star \bY J\bY^\top v=\underline{\by}^\top (J^\top \otimes (vu^\star))\underline{\by}$  is identically distributed to
$$
\bx^\top G^{\top} (J^\top \otimes (vu^\star))G\bx.
$$
So, to apply Lemma~\ref{lem:gaussianHS}, we need to bound the norms. First we have
\begin{align}
  \nonumber
  \left\| G^{\top} (J^\top \otimes (vu^\star))G\right\|_2
  &\le \|G^{1/2}\|_2^2 \| (J^\top \otimes (vu^\star))\|_2 \\
  \nonumber
  &= \|J\|_2 \|\underline{R}\|_2 \\
  \label{eq:spectral2}
  &\le \|J\|_2 \|\Phi\|_2.
\end{align}
To bound the Frobenius norm, note that $\underline{R}\preceq \|\Phi\|_{\infty} I$ so that 
\begin{align}
  \nonumber
  \MoveEqLeft
  \left\|G^{\top} (J^\top \otimes (vu^\star))G\right\|_F^2 \\
  \nonumber
  &=\Tr\left( (J^\top \otimes (vu^\star)) \underline{R} (\overline{J} \otimes (uv^\star)  ) \underline{R} \right) \\
  \nonumber
  &\le \|\Phi\|_{\infty} \Tr\left( (J^\top \otimes (vu^\star)) (\overline{J} \otimes (uv^\star)  ) \underline{R} \right) \\
  \nonumber
  &\le \|\Phi\|_{\infty}^2 \Tr\left( (J^\top \otimes (vu^\star)) (\overline{J} \otimes (uv^\star)  ) \right) \\
  \label{eq:spectralFro}
  &= \|J\|_F^2 \|\Phi\|_{\infty}^2.
\end{align}
The result now follows by applying Lemma~\ref{lem:gaussianHS} with $A = G^\top (J^\top \otimes (vu^\star))G$. Note that if $J$, $u$, and $v$ are real, then so is $A$. Similarly, if $J$ is Hermitian and $u=v$, then $A$ is Hermitian. 
\end{proof}

\subsection{A Special Result for the Sub-Gaussian Case}

\begin{lemma}
  Let Assumption~\ref{a:subgaussian} hold. Let $J\in \bbC^{N\times N}$, $u\in\bbC^{n}$, $v\in \bbC^n$ be unit vectors such that one of the following conditions holds:
  \begin{enumerate}
  \item $J\in \bbR^{N\times N}$, $u\in \bbR^n$, and $v\in\bbR^n$ or
  \item $J$ is  Hermitian and $u=v$. 
  \end{enumerate}
  Then, for any $\epsilon >0$ the following bound holds:  
    \begin{multline*}
    \bbP\left(u^\star \bY J \bY^\top v - \bbE\left[ u^\star \bY J \bY^\top v\right] > \epsilon \right) \le \\
   2\exp\left(-2^{-15} \min\left\{\frac{\epsilon^2}{\sigma^4\|J\|_F^2\|\Phi\|_{\infty}^2},\frac{\epsilon}{\sigma^2\|J\|_2\|\Phi\|_\infty}\right\}\right).
  \end{multline*}

\end{lemma}

\begin{IEEEproof}
  For all $T\ge 1$ let
  \begin{align*}
    \by_T[k] &= \sum_{\ell=-T}^T h[k-\ell]\bzeta[\ell] \\
    \bY_T &= \begin{bmatrix}\by_T[0] & \cdots & \by_T[N-1] \end{bmatrix} \\
    \hat \bPhi_T(s) &= \bY_TD(-s)AD(s)\bY_T^\top \\
    \underline{\by}_{T} &=  \begin{bmatrix}\by_T[0]^\top & \cdots & \by_T[N-1]^\top \end{bmatrix}^\top \\
     \underline{R}_T &= \bbE\left[\underline{\by}_T \underline{\by}_T^\top\right]. 
  \end{align*}

  Setting
  \begin{align*}
    \underline{\bzeta}_T & = \begin{bmatrix}\bzeta[-T]^\top & \cdots \bzeta[T]^\top \end{bmatrix}^\top \\
   G_T &= \begin{bmatrix}
    h[T] & \cdots & h[-T] \\
    \vdots \\
    h[N-1+T] & \cdots & h[N-1-T]
  \end{bmatrix} 
  \end{align*}
gives that $\underline{\by}_T = G_T \underline{\bzeta}_T$ and so $
  \underline{R}_T = G_TG_T^\top.
  $
 
  Note that
  $$
  G_{T+1} =
  \setlength\arraycolsep{2pt}
  \begin{bmatrix}
    \begin{bmatrix}
      h[T+1] \\
      \vdots \\
      h[N-1+T+1]
    \end{bmatrix}
    &
    G_T
    &
    \begin{bmatrix}
      h[-T-1] \\
      \vdots \\
      h[N-1-T-1]
    \end{bmatrix}
    \end{bmatrix}.
    $$
    It follows that $\underline{R}_T\preceq \underline{R}_{T+1}$. Furthermore, $\lim_{T\to\infty} \underline{R}_T=\underline{R}$. Thus, Lemma~\ref{lem:covFromSpec} implies that $\|G_T\|_2^2=\|\underline{R}_T\|_2 \le \|\Phi\|_{\infty}$. 

    Consider the scalar random variable
    \begin{align*}
      u^\star \bY_T J \bY_T^\top v &= \underline{\by}_T^\top (J^\top \otimes (vu^\star)) \underline{\by}_T \\
      &= \bzeta_T^\top G_T^\top (J^\top \otimes (vu^\star)) G_T \bzeta_T . 
    \end{align*}
    We can bound the deviation of this scalar random variable from its mean via the Hanson-Wright inequality with $A=G_T^\top (J^\top \otimes (vu^\star)) G_T$. Similar to  (\ref{eq:spectral2}) and (\ref{eq:spectralFro}), we have
    \begin{align*}
      \|G_T^\top (J^\top \otimes (vu^\star)) G_T\|_2 &\le \|J\|_2 \|\Phi\|_{\infty} \\
      \|G_T^\top (J^\top \otimes (vu^\star)) G_T\|_F^2 &\le \|J\|_F^2 \|\Phi\|_{\infty}^2. 
    \end{align*}
    Additionally, if $J$, $u$, and $v$ are real, then $A$ is real. If $J$ is Hermitian and $u=v$, then $A$ is also Hermitian.
    
    From Lemma~\ref{lem:subgaussianConstants} in Appendix~\ref{app:constants}, we have that $\|\bzeta_i[k]\|_{\psi_2}:=b\le 2\sigma$ for all $i$ and $k$. 
    Thus, Theorem~\ref{thm:HW} of Appendix~\ref{app:constants} implies that  
    \begin{multline*}
    \bbP\left(u^\star \bY_T J \bY_T^\top v - \bbE\left[ u^\star \bY_T J \bY_T^\top v\right] > \epsilon \right) \le \\
   2\exp\left(-2^{-15} \min\left\{\frac{\epsilon^2}{\sigma^4\|J\|_F^2\|\Phi\|_{\infty}^2},\frac{\epsilon}{\sigma^2\|J\|_2\|\Phi\|_\infty}\right\}\right).
  \end{multline*}
    Now since $\lim_{T\to\infty}\bY_T = \bY$, the result holds by dominated convergence.
  \end{IEEEproof}

\subsection{Proof of Lemma~\ref{lem:matrixConcentration}}

\const{HWmult}
\const{HWexp}
\const{HWsubgauss}

The previous two lemmas imply that there are constants $c_{\ref{HWmult}}$, $c_{\ref{HWexp}}$, $c_{\ref{HWsubgauss}}$ defined by:
\begin{align*}
  \textrm{Assumption}~\ref{a:gaussian} &\implies c_{\ref{HWmult}}=1, \quad  c_{\ref{HWexp}}=\frac{1}{8}, \quad &c_{\ref{HWsubgauss}}=1 \\
  \textrm{Assumption}~\ref{a:subgaussian} &\implies c_{\ref{HWmult}}=2, \quad c_{\ref{HWexp}}=2^{-15},\quad &c_{\ref{HWsubgauss}}=\sigma
\end{align*}
such that
\begin{multline}
  \label{eq:scalarizedConcentration}
  \bbP\left(
u^\star \bY J \bY^\top v -
\bbE\left[u^\star \bY J \bY v
\right] >\epsilon\right) \le \\
c_{\ref{HWmult}}\exp\left(-c_{\ref{HWexp}} \min\left\{\frac{\epsilon^2}{c_{\ref{HWsubgauss}}^4\|J\|_F^2\|\Phi\|_{\infty}^2},\frac{\epsilon}{c_{\ref{HWsubgauss}}^2\|J\|_2\|\Phi\|_\infty}\right\}\right),
\end{multline}
under corresponding assumptions about $J$, $u$, and $v$.

We complete the proof of Lemma~\ref{lem:matrixConcentration} by a covering argument, similar to the proof of Theorem 6.5 of \cite{wainwright2019high}. For any $\delta >0$, the Euclidean ball of dimension $n$ can be covered by a collection of at most $\left(1+\frac{2}{\delta}\right)^n$ balls with radius $\delta$. (See Example 5.8 of \cite{wainwright2019high}.) Let $\cC_n = \{w_1,\ldots,w_{Q_n}\}$ be the centers of such a covering with $\|w_i\|_2\le 1$ and $\delta = \frac{2}{9}$ so that $Q_n \le 10^n$.

For compact notation, let $\bS:=\bY J\bY^\top-\bbE\left[\bY J\bY^\top\right]$.

\paragraph*{Covering for Real $J$}
When $J$ is real, $\bS$ is also real. In this case
$$
\|\bS\|_2 = \sup_{\|u\|_2\le 1, \|v\|_2\le 1} u^\top \bS v
$$
where the supremum ranges vectors $u,v\in\bbR^n$ with Euclidean norm at most $1$. Given any $u,v\in \bbR^n$ with norm at most $1$, there are vectors $\hat u$ and $\hat v$ in $\cC_n$ such that $\|u-\hat u\|_2\le \frac{2}{9}$ and $\|v-\hat v\|_2\le \frac{2}{9}$.
\begin{align*}
  \MoveEqLeft[0]
  u^\top \bS v = \left(\hat u + (u-\hat u)\right)^\top \bS \left(\hat v + (v-\hat v)\right)) \\
  &=\hat u^\top S\hat v + (u-\hat u)^\top \bS \hat v + \hat u^\top \bS (v-\hat v) + (u-\hat u)^\top \bS (v-\hat v)  \\
  &\le \hat u^\top S \hat v + \left(\frac{4}{9}+\frac{4}{81} \right)\|\bS\|_2
    \le \hat u^\top S \hat v + \frac{1}{2}\|\bS\|_2.  
\end{align*}
The first inequality follows from the Cauchy-Schwartz inequality and submultiplicativity of the induced norm.

Maximizing the expression above on both sides leads to:
$$
\|\bS\|_2 \le \max_{\hat u,\hat v\in\cC_n} \hat u^\top \bS\hat v + \frac{1}{2} \|\bS\|_2
\implies \|\bS\|_2 \le 2\max_{\hat u,\hat v\in\cC_n} \hat u^\top \bS\hat v. 
$$

The proof is completed in this case via a union bound:
\begin{align*}
  \MoveEqLeft[0]
  \bbP\left(
  \|\bS\|_2 > \epsilon
  \right) \le \bbP\left( \max_{\hat u,\hat v \in \cC_n} \hat u^\top \bS \hat v > \epsilon/2\right) \\
          &\le \sum_{\hat u,\hat v\in\cC_n} \bbP\left( \hat u^\top \bS \hat v > \epsilon/2\right) \\
          &\le \hspace{-2pt} 10^{2n} 
  c_{\ref{HWmult}}\hspace{-1pt}\exp\left(\hspace{-2pt}-\frac{c_{\ref{HWexp}}}{4} \min\left\{\frac{\epsilon^2}{c_{\ref{HWsubgauss}}^4\|J\|_F^2\|\Phi\|_{\infty}^2},\frac{\epsilon}{c_{\ref{HWsubgauss}}^2\|J\|_2\|\Phi\|_\infty}\right\}\right)\hspace{-2pt}.
\end{align*}
The final inequality arises because $\cC_n\times \cC_n$ has at most $10^{2n}$ elements.

\paragraph*{Covering for Hermitian $J$}
When $J$ is Hermitian, $\bS$ is Hermitian as well. In this case
$$
\|\bS\|_2 = \sup_{\|u\|_2\le 1} |u^\star \bS u|
$$
where the supremum ranges over the unit ball of $\bbC^n$. The unit ball of $\bbC^n$ can be identified with the unit ball of $\bbR^{2n}$: If $u=v+jw$ with $v$ and $w$ real vectors, we have that $\|u\|_2\le 1$ if and only if $\left\|\begin{bmatrix}v^\top & w^\top \end{bmatrix}^\top\right\|_2\le 1$.

Let $\cC_{2n}$ be the centers of a $\frac{2}{9}$-covering of the unit ball of $\bbR^{2n}$ and define a $\frac{2}{9}$-covering of the unit ball of $\bbC^n$ by:
$$
\hat\cC_n = \left\{v + jw \middle| \begin{bmatrix}v^\top & w^\top \end{bmatrix}^\top \in\cC_{2n}\right\}.
$$
Since $\cC_{2n}$ has at most $10^{2n}$ elements, $\hat\cC_n$ also has at most $10^{2n}$ elements.

Similar to the real case, we have that for  all $\|u\|_2\le 1$, there exists $\hat u\in\hat \cC_n$ such that $\|u-\hat u\|_2\le \frac{2}{9}$. Then we have:
\begin{align*}
  \MoveEqLeft[0]
  \left|u^\star \bS u\right| = \left|\left(\hat u + (u-\hat u)\right)^\star \bS \left(\hat u + (u-\hat u)\right))\right| \\
  &\le \left|\hat u^\star S \hat u\right| + \frac{1}{2}\|\bS\|_2.  
\end{align*}
After maximizing both sides and re-arranging, we get $\|\bS\|_2 \le 2\max_{\hat u\in\hat \cC_n}\left|\hat u^\star \bS \hat u \right| $.

The proof is completed in this case by a union bound argument:
\begin{align*}
  \MoveEqLeft[0]
  \bbP\left(\|\bS\|_2 > \epsilon \right)
  \le
    \bbP\left(
    \max_{\hat u\in\hat\cC_n}\left| \hat u^\star \bS \hat u\right| > \epsilon/2
    \right) \\
  &\le \sum_{\hat u\in \hat\cC_n} \bbP\left(
    \left|\hat u^\star \bS \hat u\right| > \epsilon/2
    \right) \\
  &\le \sum_{\hat u\in \hat\cC_n} \left(\bbP\left(
    \hat u^\star \bS \hat u > \epsilon/2
    \right) +
\bbP\left(
    \hat u^\star (-\bS) \hat u > \epsilon/2
    \right)
    \right)
  \\
  &\le
 2\cdot 10^{2n} 
    c_{\ref{HWmult}}\hspace{-1pt}e^{\hspace{-2pt}-\frac{c_{\ref{HWexp}}}{4} \min\left\{\frac{\epsilon^2}{c_{\ref{HWsubgauss}}^4\|J\|_F^2\|\Phi\|_{\infty}^2},\frac{\epsilon}{c_{\ref{HWsubgauss}}^2\|J\|_2\|\Phi\|_\infty}\right\}}.
\end{align*}
\hfill\IEEEQED

\section{Proof of Theorem~\ref{thm:genConvergence}}
\label{app:genConvergence}

We prove parts \ref{it:concentration}), \ref{it:worst}), \ref{it:bias}), and \ref{it:fullWorst}). The proof of \ref{it:fullPointwise}) is omitted, since it is similar to the proof of \ref{it:fullWorst}). 

\subsection{Proof of \ref{it:concentration})}
Note that $\hat\Phi(s)=\bY J \bY^\top$ where $J=D(s)^\star AD(s)$. Since $D(s)$ is unitary, we have $\|J\|_2=\|A\|$ and $\|J\|_F=\|A\|_F$.  Since $A\in\bbR^{N\times N}$ is symmetric, $J$ is Hermitian and so Lemma~\ref{lem:matrixConcentration} implies that  
\begin{align*}
  \MoveEqLeft[0]
    \bbP\left(
      \left\|\hat \bPhi(s) - \bbE\left[\hat\bPhi(s)\right]\right\|_2 > \epsilon
  \right) \\
  &\le
    10^{2n} c_{\ref{ConcentrationMult}}e^{-c_{\ref{ConcentrationExp}} \min\left\{\frac{\epsilon^2}{c_{\ref{ConcentrationSubGauss}}^4\|A\|_F^2\|\Phi\|_{\infty}^2},\frac{\epsilon}{c_{\ref{ConcentrationSubGauss}}^2\|A\|_2\|\Phi\|_\infty}\right\}} \\
  &\le  10^{2n} c_{\ref{ConcentrationMult}}e^{-\frac{c_{\ref{ConcentrationExp}}}{\max\{\|A\|_2,\|A\|_F^2\}} \min\left\{\frac{\epsilon^2}{c_{\ref{ConcentrationSubGauss}}^4\|\Phi\|_{\infty}^2},\frac{\epsilon}{c_{\ref{ConcentrationSubGauss}}^2\|\Phi\|_\infty}\right\}}.
\end{align*}
The right side is at most $\delta$ if and only if $\frac{1}{\max\{\|A\|_2,\|A\|_F^2\}}\ge \xi$.

\subsection{Proof of \ref{it:worst})}

For this proof, let $\bM(s)=\hat\bPhi(s)-\bbE\left[\hat\bPhi(s)\right]$ and $\bC[k]=\bY B[k] \bY^\top -\bbE\left[\bY B[k] \bY^\top \right]$ so that (\ref{eq:expandedEstimator}) implies
$$
\bM(s)=\sum_{k=-\hat N+1}^{\hat N-1} e^{-j2\pi sk}\bC[k].
$$
Here we also used that $B[k]=0$ for $|k|\ge \hat N$.

Note that the quantity we must bound can be expressed as $\|\bM\|_{\infty}=\sup_{|s|\le \frac{1}{2}} \|\bM(s)\|_2$. 
$$
\|\bM\|_{\infty} = \sup_{\|u\|=1,|s|\le \frac{1}{2}} \left|u^\star \bM(s) u \right|
$$
where $u$ ranges over unit vectors in $\bbC^n$.

To eliminate the supremum over $s$, we will use a covering argument. Fix a covering of $\left[-\frac{1}{2},\frac{1}{2}\right]$ with intervals of length $\frac{1}{\pi \hat N^2}$, which correspond to balls of radius $\frac{1}{2\pi \hat N^2}$. Let $\hat\cC$ denote the corresponding centers of the intervals. Note that $\hat \cC$ can be chosen to have at most $1+\pi \hat N^2$ elements.   

For any $s\in \left[-\frac{1}{2},\frac{1}{2}\right]$, there is an $\hat s\in\hat \cC$ such that $|s-\hat s|\le \frac{1}{2\pi \hat N^2}$. Then we can bound:
\begin{align*}
  \MoveEqLeft
  \|\bM(s)\|_2 \\
  &= \|\bM(\hat s) +\bM(s)-\bM(\hat s)\|_2 \\
  &\le \|\bM(\hat s)\|_2 +\|\bM(s)-\bM(\hat s)\|_2 \\
  &= \|\bM(\hat s)\|_2 + \left\| \sum_{|k|<\hat N} \left(e^{-j2\pi sk}-e^{-j2\pi \hat s k} \right) \bC[k]\right\|_2 \\
  &\le \|\bM(\hat s)\|_2 +  \sum_{|k|<\hat N} \left| e^{-j2\pi sk}-e^{-j2\pi \hat s k} \right| \|\bC[k]\|_2 \\
  &\le \|\bM(\hat s)\|_2 +  2\pi |s-\hat s| \sum_{|k|<\hat N} |k| \|\bC[k]\|_2 \\
  &\le \|\bM(\hat s)\|_2 + 2\pi |s-\hat s| \left(\max_{|k|<\hat N} \|\bC[k]\|_2\right) \sum_{|i|<\hat N} |i| \\
  &\le \|\bM(\hat s)\|_2 + 2\pi |s-\hat s|\hat N^2 \left(\max_{|k|<\hat N} \|\bC[k]\|_2\right) \\
  &\le \|\bM(\hat s)\|_2 + \max_{|k|<\hat N} \|\bC[k]\|_2.
\end{align*}
The final inequality follows from the choice of $\hat s$.

Taking suprema over $s$ shows that
$$
\|\bM\|_\infty\le \max_{\hat s\in \hat\cC} \|\bM(\hat s)\|_2 + \max_{|k|<\hat N} \|\bC[k]\|_2.
$$
Thus, we can use a union bounding argument to show:
\begin{align}
  \nonumber
  \MoveEqLeft[0]
  \bbP\left(\|\bM\|_\infty > \epsilon \right) \\
  \nonumber
  &\le
    \bbP\left(
    \max_{\hat s \in\hat \cC} \|\bM(\hat s)\|_2 > \frac{\epsilon}{2}
    \right) + \bbP\left(
    \max_{|k|<\hat N} \|\bC[k]\|_2 >\frac{\epsilon}{2}
    \right) \\
  \label{eq:worstUnion}
  &\le \sum_{\hat s\in\hat \cC} \bbP\left(
    \|\bM(\hat s)\|_2 > \frac{\epsilon}{2}
    \right)+\sum_{|k|<\hat N}\bbP\left( \|\bC[k]\|_2 >\frac{\epsilon}{2} \right).
\end{align}
So, to make the overall sum at most $\delta$, it suffices that each individual summation is at most $\delta/2$. 

The first sum on the right of (\ref{eq:worstUnion}) can be bounded using part \ref{it:concentration}), the assumption that $g\ge \max\{\|A\|_2,\|A\|_F^2\}$, and the fact that $|\hat \cC|\le 1+\pi \hat N^2\le 5\hat N^2$: 
\begin{align*}
  \MoveEqLeft
\sum_{\hat s\in\hat \cC} \bbP\left(
    \|\bM(\hat s)\|_2 > \frac{\epsilon}{2}
  \right)
\\
  &\le 5\hat N^2 10^{2n} c_{\ref{ConcentrationMult}}e^{-\frac{c_{\ref{ConcentrationExp}}}{g} \min\left\{\frac{\epsilon^2}{4c_{\ref{ConcentrationSubGauss}}^4\|\Phi\|_{\infty}^2},\frac{\epsilon}{2c_{\ref{ConcentrationSubGauss}}^2\|\Phi\|_\infty}\right\}}. 
\end{align*}

To make the right side at most $\delta/2$, it suffices to have $\frac{1}{g}\ge \alpha(\epsilon/2)\left(\log(5\hat N^2)+\beta(\delta/2)\right)$.

  To bound the second sum on the right of (\ref{eq:worstUnion}), recall that $g\ge \|B[k]\|_2 $ and $g\ge \|B[k]\|_F^2$ for $|k| <\hat N$. Since $B[k]\in\bbR^{N\times N}$ and there are $2\hat N -1 < 2\hat N$ terms in the sum, Lemma~\ref{lem:matrixConcentration} implies that
  \begin{align*}
    \MoveEqLeft
    \sum_{|k|<\hat N}\bbP\left( \|\bC[k]\|_2 >\frac{\epsilon}{2} \right)
    \\
    &\le 2\hat N 10^{2n} c_{\ref{ConcentrationMult}}e^{-\frac{c_{\ref{ConcentrationExp}}}{g} \min\left\{\frac{\epsilon^2}{4c_{\ref{ConcentrationSubGauss}}^4\|\Phi\|_{\infty}^2},\frac{\epsilon}{2c_{\ref{ConcentrationSubGauss}}^2\|\Phi\|_\infty}\right\}}
  \end{align*}
  To make the right side at most $\delta/2$, it suffices to have $\frac{1}{g}\ge \alpha(\epsilon/2)\left(\log(2\hat N)+\beta(\delta/2)\right)$, which is true if $\frac{1}{g}\ge \alpha(\epsilon/2)\left(\log(5\hat N^2)+\beta(\delta/2)\right)$.

  \subsection{Proof of \ref{it:bias})}
Since $b[k]\in [0,1]$ and $b[k] \ge 1-\frac{\epsilon}{2\|R\|_1}$, it follows that $\left|1-b[k]\right|\le \frac{\epsilon}{2\|R\|_1}$.
  Using the triangle inequality followed by the conditions on $b[k]$ gives:
  \begin{align*}
    \MoveEqLeft
  \left\| \Phi(s)-
    \bbE\left[\hat\bPhi(s) \right] \right\|_2 \le \sum_{k=-\infty}^{\infty}|1-b[k]| \|R[k]\|_2 \\
    &\le \frac{\epsilon}{2\|R\|_1} \sum_{|k|<\hat M} \|R[k]\|_2 + \sum_{|\ell| \ge \hat M} \|R[\ell]\|_2 \\
    &\le \frac{\epsilon}{2} + \frac{\epsilon}{2}.
  \end{align*}

\subsection{Proof of \ref{it:fullWorst})}
Maximizing both sides of the triangle inequality from (\ref{eq:triangle}) gives
\begin{multline*}
  \sup_{s\in\left[-\frac{1}{2},\frac{1}{2}\right]}\|\Phi(s)-\hat\bPhi(s)\|_2\le
  \sup_{s\in\left[-\frac{1}{2},\frac{1}{2}\right]}\left\|\Phi(s)-\bbE\left[\hat\bPhi(s)\right]\right\|_2
\\+
\sup_{s\in\left[-\frac{1}{2},\frac{1}{2}\right]}\
\left\|\hat \bPhi(s)-\bbE\left[\hat\bPhi(s)\right]\right\|_2.
\end{multline*}
Assuming the conditions of \ref{it:bias}) implies that $\sup_{s\in\left[-\frac{1}{2},\frac{1}{2}\right]}\left\|\Phi(s)-\bbE\left[\hat\bPhi(s)\right]\right\|_2 \le \epsilon$ surely. So, if the left side is greater than $2\epsilon$, we must have that $\sup_{s\in\left[-\frac{1}{2},\frac{1}{2}\right]}\left\|\Phi(s)-\bbE\left[\hat\bPhi(s)\right]\right\|_2 > \epsilon$, which holds with probability at most $\delta$ because the conditions of \ref{it:worst}) are also assumed. 
\hfill\IEEEQED

\subsection{Proof of Corollary~\ref{cor:gen}}
\label{app:corGen}

The first two parts are a direct consequence of Theorem~\ref{thm:genConvergence} and the inverse formula
$\alpha^{-1}(t) = c_{\ref{ConcentrationSubGauss}}^2\|\Phi\|_\infty\max\left\{t^{-1},t^{-1/2}\right\}$. The third part bounds the bias in the important special case that the autocovariance decays geometrically, and is found by direct calculation. 

Using $a$ and $b$ as defined in part~\ref{it:unknownNorm} gives
\begin{align*}
  \left\|\hat\bPhi-\Phi\right\|_{\infty} &\le b +  \left\|\hat \bPhi-\bbE\left[\hat\bPhi\right]\right\| \\
                                         &\overset{\ref{it:worstCor})}\le b + a \|\Phi\|_{\infty} \\
  &\le b+ a\left(\|\hat\bPhi\|_{\infty}+\left\|\Phi-\hat\bPhi\right\|_{\infty} \right)
\end{align*}
The result now follows by re-arranging. \hfill\IEEEQED

\section{Proofs for Specific Estimators}
\label{app:specific}
For all the specific estimators, we utilize Theorem~\ref{thm:genConvergence}. To this end, we derive upper bounds on $\|A\|_2$, $\|A\|_F^2$, $\|B[k]\|_2$, and $\|B[k]\|_F^2$ and derive sufficient conditions on $b[k]$ to achieve the desired bias. 

\subsection{Proof of Proposition~\ref{prop:biasedPeriodogram} on Biased Periodograms}

For all $|k|<N$, we have $b[k] = 1-\frac{|k|}{N}\in [0,1]$. Then $b[k]\ge 1-\frac{\epsilon}{2\|R\|_1}$ if and only if $|k|\le  \frac{\epsilon N}{2\|R\|_1}$. So, to have $b[k]\ge 1-\frac{\epsilon}{2\|R\|_1}$ for all $|k|<\hat M(\epsilon)$, it suffices to have $\hat M(\epsilon) \le \frac{N\epsilon}{2\|R\|_1}$.  
\hfill\IEEEQED

\subsection{Proof of Proposition~\ref{prop:unbiasedPeriodogram} on Unbiased Periodograms}
For all $|k|<N$, we have $b[k]=1$. So to have $b[k]\ge 1-\frac{\epsilon}{2\|R\|_1}$ for all $|k|<\hat M(\epsilon)$ it suffices that $N\ge \hat M(\epsilon)$.
\hfill\IEEEQED

\subsection{Proof of Theorem~\ref{thm:BTconvergence} on Blackman-Tukey Estimators}
\label{app:BT}

To prove \ref{it:BTconcentration}) it suffices to show $\|A\|_2 \le \frac{(2M-1)}{N}$ and $ \|A\|_F^2\le \frac{(2M-1)}{N}$. 

    Since $A$ is symmetric, the induced norm can be expressed as $\|A\|_2 = \sup_{\|u\|_2\le 1} |u^\top Au|$, where the supremum ranges over real-valued vectors with norm at most $1$. Given any vector $u\in\bbR^N$, we have
  \begin{align*}
    u^\top NA u = w[0]u^\top u + \sum_{i=1}^{M-1}(w[-i]+w[i])\sum_{k=i}^{N-1}u[k-i]u[k]. 
  \end{align*}
  So, if $\|u\|_2\le 1$, 
  it follows that
  \begin{align*}
    |u^\top NA u |
                   &\le 1 + 2\sum_{i=1}^{M-1}\sum_{k=i}^{N-1}|u[k-i]||u[k]| \\
                &\le 1+\sum_{i=1}^{M-1}\sum_{k=i}^{N-1}\left(|u[k-i]|^2+|u[k]|^2\right) \\
                &\le 1+2(M-1)
  \end{align*}
  The bound on $\|A\|_2$ follows by dividing by $N$.

  The Frobenius norm can be bounded as:
  \begin{align*}
    N^2 \|A\|_F^2 &= \sum_{k=-M+1}^{M-1}w[k]^2 (N-|k|) \\
                  &\le \sum_{k=-M+1}^{M-1}(N-|k|)
                    \le N(2M-1)
  \end{align*}
  The upper bound on the Frobenius norm follows by dividing by $N^2$, and \ref{it:BTconcentration}) is proved. 

  Now we prove \ref{it:BTworst}). We have that $B[k] = 0$ for $|k| \le M$, so set $\hat N = M$.

    Direct calculation gives: 
    \begin{align*}
      \|B[k]\|_2 &= \frac{|w[k]|}{N}\le \frac{1}{N} \\
      \|B[k]\|_F^2 &= \frac{|w[k]|^2(N-|k|)}{N^2}\le \frac{1}{N}. 
    \end{align*}
    So, we can take $g = \frac{2M-1}{N}$.

Now we prove \ref{it:BTbias}). 
  Note that 
$$
b[k] = \begin{cases}
  \frac{(N-|k|)w[k]}{N} & |k| < M \\
  0 & |k| \ge M.
\end{cases}
$$
So, if $0\le w[k] \le 1$, we have $0\le b[k] \le 1$ as well. Furthermore, for $|k|<M$, we  have that $b[k]\ge 1- \frac{\epsilon}{2\|R\|_1}$ if and only if
\begin{equation}
  \label{eq:BTWeightBound}
w[k]\ge \frac{1-\frac{\epsilon}{2\|R\|_1}}{1-\frac{|k|}{N}}.
\end{equation}

To ensure that \eqref{eq:BTWeightBound} can be satisfied with $|w[k]|\le 1$, the right side must be bounded above by $1$, which occurs
if and only if $|k|\le \frac{N\epsilon}{2\|R\|_1}$. Thus, if $\hat M(\epsilon)\le \frac{N\epsilon}{2\|R\|_1}$, the bias bound from \ref{it:BTbias})
will be achieved as long as (\ref{eq:BTWeightBound}) holds for $|k|<\hat M$ and $w[k] \in [0,1]$ for $|k|\ge \hat M(\epsilon)$. 
\hfill\IEEEQED

\subsection{Proof of Theorem~\ref{thm:BartlettConvergence} on Bartlett Estimators}

Part \ref{it:BartlettConcentration}) follows because $\|A\|_2=\|A\|_F^2 = \frac{M}{N}$, by direct calculation.

Now we prove \ref{it:BartlettWorst}). We have $\hat N=M$. For $|k|<M$ direct calculation gives
\begin{align*}
  \|B[k]\|_2 &= \frac{1}{N} \\
  \|B[k]\|_F^2& = \frac{N-L|k|}{N^2}\le \frac{1}{N}.
\end{align*}
So, we can take $g=\frac{M}{N}$. 

Now we prove \ref{it:BartlettBias}). For $|k| < M$ we have
$$
b[k] =\frac{L(M-|k|)}{LM}=1-\frac{|k|}{M}. 
$$
Let $\hat \epsilon = \frac{\epsilon}{2\|R\|_1}$. We see that $b[k] \ge 1-\hat\epsilon$ if and only if $|k| \le M\hat\epsilon$. So, to ensure that $b[k]\ge 1-\hat \epsilon$ for all $|k|< \hat M$, it suffices to have $\hat M(\epsilon) \le M\hat \epsilon$.
\hfill\IEEEQED

\subsection{Proof of Theorem~\ref{thm:WelchConvergence} on Welch Estimators}

First we prove \ref{it:WelchConcentration}). It suffices to show that $\|A\|_2\le \frac{1+2\frac{M}{K}}{S}$ and $\|A\|_F^2\le \frac{1+2\frac{M}{K}}{S}$. 

 Without loss of generality, assume that $\|v\|_2=1$. Indeed, the normalization in (\ref{eq:WelchFT}) implies that the window $v/\|v\|_2$ leads to the same estimator as $v$.

  For $k=0,\ldots,\ceil*{\frac{M}{K}}-1$, let $\cI_k = \left\{i\in \{0,\ldots,S-1\} | i\mod \ceil*{\frac{M}{K}}=k\right\}$. The sum in (\ref{eq:AWelch}) can be re-grouped to give:
  \begin{align}
    \nonumber
  SA &= \sum_{k=0}^{\ceil*{\frac{M}{K}}-1}\sum_{i\in\cI_k}
\setlength\arraycolsep{2pt}
  \begin{bmatrix}
    0_{iK\times iK}  \\
    & vv^\top \\
    && 0_{(N-iK-M)\times (N-iK-M)}
  \end{bmatrix} \\
    \label{eq:regroup}
    &=: \sum_{k=0}^{\ceil*{\frac{M}{K}}-1} C_k
  \end{align}

  The matrices, $C_k$, are block diagonal with blocks either $vv^\top$ or zero matrices. Indeed, if $p<q$ are both in $\cI_k$, then $qK-pK\ge M$, and the $vv^\top$ blocks in the $p$th and $q$th matrices in the original sum from (\ref{eq:AWelch}) have size $M\times M$. As a result, there is no overlap in the non-zero portions of these matrices. Now, since $v$ is a unit vector, we have that $\|C_k\|_2\le 1$. So, the triangle inequality implies that $\|SA\|_2\le \ceil*{\frac{M}{K}}$. The bound on $\|A\|_2$ follows by dividing by $S$.

  To bound $\|A\|_F^2$, first note that we can rewrite:
  $$
  SA= \sum_{i=0}^{S-1}\begin{bmatrix}
      0_{iK\times 1} \\
      v \\
      0_{(N-iK-M)\times 1}
    \end{bmatrix}
    \begin{bmatrix}
      0_{iK\times 1} \\
      v \\
      0_{(N-iK-M)\times 1}
    \end{bmatrix}^\top
    $$
    As a result, we have that
    \begin{align*}
      \|SA\|_F^2 
      &= S\\
      &\hspace{-30pt}
                   + 2\sum_{p=0}^{S-2}\sum_{q=p+1}^{S-1}
                  \left( \begin{bmatrix}
      0_{pK\times 1} \\
      v \\
      0_{(N-pK-M)\times 1}
    \end{bmatrix}^\top
      \begin{bmatrix}
      0_{qK\times 1} \\
      v \\
      0_{(N-qK-M)\times 1}
    \end{bmatrix}
      \right)^2 \\
      &\le S+2(S-1)\left(\ceil*{\frac{M}{K}} - 1\right).
    \end{align*}
    The inequality follows because the vectors in the inner products are all unit vectors, and so the inner products have magnitude at most $1$ by the Cauchy-Schwartz inequality. Furthermore, if $q\ge \ceil*{\frac{M}{K}}$, then $qK-pK \ge M$, and so the non-zero portions of the corresponding vectors have no overlap. As a result, at most $\ceil*{\frac{M}{K}}-1$ terms in the inner sum can be non-zero. The bound on $\|A\|_F^2$ follows by dividing by $S^2$ and simplifying.

    Now we prove \ref{it:WelchWorst}). First note  that $\hat N=M$. We will show that $\|B[k]\|_2\le \frac{1}{S}$ and $\|B[k]\|_F^2 \le \frac{1}{S}$ for $|k|<M$. Thus, in this case, we can take $g=\frac{1+2\frac{M}{K}}{S}$.

    To bound $\|B[k]\|_2$, 
    we first analyze the diagonal of $SA$. Each entry on the diagonal is of the form
    \begin{equation}
      \label{eq:WelchDiagonal}
    SA_{p,p}=\sum_{i\in \cJ_p} v[i]^2 \overset{\|v\|_2=1}{\le} 1,
    \end{equation}
    where $\cJ_p\subset \{0,\ldots,M-1\}$. 

    Now, for any $p\ne q$, positive semidefiniteness implies that  $(SA_{p,q})^2 \le (SA_{p,p}) (SA_{q,q}) \le 1$. 
    It  now follows that $\|B[k]\|_2\le \frac{1}{S}$ for all $|k|<M$.

    To bound $\|B[k]\|_F^2$, symmetry of $A$ combined with (\ref{eq:expandedMatrices}) and (\ref{eq:expandedNorms}) gives for $|k|<M$:
    \begin{align*}
      \|SB[k]\|_F^2 &=\sum_{i=|k|}^{N-1} (SA_{i,i-|k|})^2 \\
                    &\overset{SA\succeq 0}{\le} \sum_{i=|k|}^{N-1} (SA_{i,i})(SA_{i-|k|,i-|k|}) \\
                    &\overset{\eqref{eq:WelchDiagonal}}{\le} \sum_{i=|k|}^{N-1} SA_{i,i}\\
      &\le \sum_{i=0}^{N-1}SA_{i,i}=S.
    \end{align*}
    Dividing both sides by $S^2$ gives $\|B[k]\|_F^2 \le \frac{1}{S}$.

Now we prove \ref{it:WelchBias}). To state the conditions for the original $v$, we do not assume that $v$ is normalized, but assume that $v[k]\ge 0$.  So, in this case
$$
b[k]=\begin{cases}
  \sum_{i=|k|}^{M-1}\frac{v[i-|k|]v[i]}{\|v\|_2^2} & |k| < M \\
  0 & |k| \ge M.
  \end{cases}
  $$
  So, it suffices to have  $M\ge \hat M(\epsilon)$ and for $|k|<\hat M(\epsilon)$ to have
  $$
  \sum_{i=|k|}^{M-1}\frac{v[i-|k|]v[i]}{\|v\|_2^2}\ge 1-\frac{\epsilon}{2\|R\|_1} 
  $$
  \hfill\IEEEQED

\section{Tracking Constants in Concentration Bounds}
\label{app:constants}
The goal of this appendix is to derive explicit expressions arising in the concentration bounds used in the paper. In particular, an explicit bound for the constant in the Hanson-Wright inequality is derived. 

Let $\psi_2(x)=e^{x^2}-1$ and define the $\psi_2$-Orlicz norm by:
$$
  \|\bx\|_{\psi_2} = \inf\left\{t>0 \middle| \bbE\left[e^{\bx^2/t^2}-1 \right]\le 1\right\}.
$$

\begin{lemma}
  \label{lem:subgaussianConstants}
Let $\bx$ be a scalar zero-mean random variable.
\begin{itemize}
\item If $\|\bx\|_{\psi_2}\le b$, then
  \begin{subequations}
    \begin{align}
      \label{eq:orcliczProb}
      \bbP\left(|\bx| > t\right) &\le 2e^{-t^2/b^2} \quad \forall t\ge 0 \\ 
      \label{eq:orliczEvenMoments}
      \bbE\left[\bx^{2k}\right] & \le 2b^{2k} k! \quad \forall k\ge 0 \\
      \label{eq:orliczMGF}
      \bbE\left[
      e^{\lambda \bx}
      \right] &\le e^{4\lambda^2 b^2} \quad \forall \lambda\in\bbR \\
      \label{eq:orliczCenteredVar}
      \bbE\left[\left(\bx^{2}-\bbE[\bx^2]\right)^k\right] & \le 2(2b^2)^{k} k! \quad \forall k\ge 0
      \\
      \hspace{-30pt}
      \label{eq:squareSubexp}
      \bbE\left[
      \exp\left(\lambda (\bx^2-\bbE[\bx^2])\right)
      \right] &\le \exp((4b^2)^2\lambda^2) \quad \forall |\lambda| \le \frac{1}{4b^2} 
    \end{align}
  \end{subequations}
  \item If $\bbE\left[e^{\lambda \bx}\right]\le e^{\frac{\lambda^2\sigma^2}{2}}$ for all $\lambda\in\bbR$, then
  \begin{subequations}
  \begin{align}
    \label{eq:quadExp}
    \bbE\left[\exp\left(\frac{\lambda \bx^2}{2\sigma^2}\right)\right] & \le\frac{1}{\sqrt{1-\lambda}} \quad \forall\lambda\in[0,1) \\
    \label{eq:subgaussianOrlicz}
    \|\bx\|_{\psi_2}&\le \sqrt{\frac{8}{3}}\sigma \le 2\sigma \\
                          \label{eq:subgaussianVar}
    \bbE\left[\bx^2\right]&\le \sigma^2 
  \end{align}
\end{subequations}
\end{itemize}
\end{lemma}

\begin{IEEEproof}
  Inequality (\ref{eq:orcliczProb}) follows from Proposition 2.5.2 of \cite{vershynin2018high}.

  For (\ref{eq:orliczEvenMoments}), the inequality is trivial at $k=0$. For $k\ge 1$, we have:
  \begin{align*}
    \bbE\left[\bx^{2k}\right]
    &=\int_0^{\infty}
      \bbP\left(\bx^{2k}>t \right)
      dt \\
    &=\int_0^{\infty}
      \bbP\left(
      |\bx|> t^{\frac{1}{2k}}
      \right)
      dt \\
    &\overset{(\ref{eq:orcliczProb})}{\le} 2\int_{0}^{\infty}\exp\left(-\frac{t^\frac{1}{k}}{b^2}\right)dt \\
    &\overset{s=\frac{t^{\frac{1}{k}}}{b^2}}{=}2kb^{2k} \int_0^{\infty}e^{-s}s^{k-1}ds \\
    &=2b^{2k}k!
  \end{align*}

   
  A similar calculation for (\ref{eq:orliczEvenMoments}) is done in the proof of Proposition 2.5.2 in \cite{vershynin2018high}. We separate the even moments, since a tighter bound can be obtained in this case.

  To prove (\ref{eq:orliczMGF}), we follow the methodology from the proof of Proposition 2.5.2 in \cite{vershynin2018high}.
  For $|\lambda|\le \frac{1}{\sqrt{2}b}$ we have:
  \begin{align*}
    \bbE\left[
    e^{\lambda^2 \bx^{2}}
    \right] &=1+\sum_{k=1}^{\infty}\frac{\lambda^{2k}}{k!}\bbE[\bx^{2k}] \\
            &\overset{~(\ref{eq:orliczEvenMoments})}{\le}1+2\sum_{k=1}^{\infty}\lambda^{2k}b^{2k}\\
            &=1+2\frac{\lambda^2 b^2}{1-\lambda^2b^2}\\
            &\le 1+ 4\lambda^2b^2
              \le e^{4\lambda^2 b^2}
  \end{align*}

  Then using $e^{x}\le x + e^{x^2}$, which holds for all $x$, we  have that (\ref{eq:orliczMGF}) holds for $|\lambda|\le \frac{1}{\sqrt{2}b}$.
  
  For $|\lambda| > \frac{1}{\sqrt{2}b}$, we use that
  $$
  \lambda \bx = \left(\frac{\sqrt{2}\bx}{b}\right)\left(\frac{\lambda b}{\sqrt{2}}\right) \le \frac{\bx^2}{b^2}+\frac{\lambda^2 b^2}{4}.
  $$
  So, in this case we also have
  \begin{align*}
    \bbE\left[
    e^{\lambda \bx}
    \right] \le 2e^{\frac{\lambda^2 b^2}{4}} \le e^{4\lambda^2 b^2}.
  \end{align*}
  The final inequality follows because
  $
  e^{\frac{15}{4}\lambda^2 b^2} \ge e^{\frac{15}{8}}>2.
  $

  Inequality (\ref{eq:orliczCenteredVar}) is trivial at $k=0$, so assume that $k\ge 1$. The triangle inequality, followed by (\ref{eq:orliczEvenMoments}) gives
  \begin{align*}
    \|\bx^2-\bbE[\bx^2]\|_k &\le \|\bx^2\|_k + \|\bbE[\bx^2]\|_k \\
    &= \|\bx^2\|_k + \bbE[\bx^2] \\
    &\le b^2\left((2k!)^{1/k} + 2\right).
  \end{align*}
  For $k=1,\ldots,5$ it can be checked that $(2k!)^{1/k}\ge 2$. For $k\ge 6$, the Stirling bound $k!\ge (k/e)^k$ implies $(2k!)^{1/k}\ge 2$. 

  So, for all $k\ge 1$ we have
  $$
  \|\bx^2-\bbE[\bx^2]\|_k \le  2b^2 (2k!)^{1/k}.
  $$
  Raising both sides to the $k$th power proves (\ref{eq:orliczCenteredVar}).

  To show (\ref{eq:squareSubexp}), note  that for all $|\lambda| \le \frac{1}{4b^2}$, we have
  \begin{align*}
    \bbE\left[
    \exp\left(\lambda (\bx^2-\bbE[\bx^2])\right)
    \right] &= 1 + \sum_{k=2}^{\infty} \frac{\lambda^k}{k!} \bbE\left[
              (\bx^2-\bbE[\bx^2])^k
              \right] \\
            &\overset{~(\ref{eq:orliczCenteredVar})}{\le}
              1+2\sum_{k=2}^{\infty}\left(2\lambda b^2\right)^k \\
            &=1+2\frac{(2\lambda b^2)^2}{1-2\lambda b^2} \\
            &\le 1+4 (2\lambda b^2)^2
              \le \exp(4 (2\lambda b^2)^2).
  \end{align*}

  Inequality (\ref{eq:quadExp}) is proved in Appendix A of \cite{wainwright2019high}.

  For (\ref{eq:subgaussianOrlicz}), set $\lambda=\frac{3}{4}$ so that $\frac{1}{\sqrt{1-\lambda}}=2$. Set $t=\sqrt{\frac{8}{3}}\sigma$, so that (\ref{eq:quadExp}) implies that $\bbE\left[e^{\bx^2/t^2}\right]\le 2$. Thus (\ref{eq:subgaussianOrlicz}) holds. 

  To prove (\ref{eq:subgaussianVar}), note that
  \begin{align*}
    \bbE\left[e^{\lambda \bx}\right] &= 1+\lambda^2\left(\frac{\bbE[\bx^2]}{2} + O(\lambda)\right) \\
                                     &\le e^{\frac{\sigma^2\lambda^2}{2}} \\
                                     &= 1+\lambda^2\left(\frac{\sigma^2}{2}+O(\lambda^2) \right)
  \end{align*}
  When $\lambda\ne 0$, re-arranging gives $\bbE\left[\bx^2\right]\le \sigma^2+O(\lambda)$. Taking the limit $\lambda \to 0$  proves (\ref{eq:subgaussianVar}). 
\end{IEEEproof}

\begin{lemma}
  Let $\bx_i$ be independent zero-mean sub-Gaussian random variables with $\bbE\left[e^{\lambda \bx_i}\right]\le e^{\frac{\lambda^2\sigma^2}{2}}$ for all $i=1,\ldots,n$. If $\bx = \begin{bmatrix}\bx_1 & \cdots & \bx_n\end{bmatrix}^\top$, then the covariance is a diagonal matrix that satisfies
  $$
  \bbE\left[\bx\bx^\top\right]\preceq \sigma^2 I_n. 
  $$
\end{lemma}

\begin{IEEEproof}
  Diagonality is immediate because $\bbE[\bx_i\bx_j] = 0$ for $i\ne j$. Then the bound on the diagonal follows from (\ref{eq:subgaussianVar}). 
\end{IEEEproof}

A random variable, $\bx$ is $(\nu,\alpha)$-subexponential if for all $|\lambda| < \frac{1}{\alpha}$, the following bound holds:
$$
\bbE\left[
  \exp\left(\lambda (\bx-\bbE[\bx])
  \right)
\right] \le e^{\frac{\lambda^2 \nu^2}{2}}.
$$
(Here, $\alpha$ is just a number, not to be confused with the specific quantity used for the bounds, $\alpha(\epsilon)$, defined in (\ref{eq:collectedConstants}).)

For all $t\ge 0$, a $(\nu,\alpha$)-subexponential random variable satisfies:
\begin{equation}
  \label{eq:subExpTail}
  \bbP\left(\bx-\bbE[\bx]>t\right) \le \exp\left(
    -\frac{1}{2}\min\left\{
      \frac{t^2}{\nu^2},
      \frac{t}{\alpha}
    \right\}
  \right)
\end{equation}
See Proposition 2.9 of \cite{wainwright2019high}.

\begin{lemma}
  Let $\bx_i$ be independent scalar-valued zero-mean random variables such that $\|\bx_i\|_{\psi_2}\le b$ for all $i=1,\ldots,n$, and let $a=\begin{bmatrix}a_1 & \cdots & a_n\end{bmatrix}^\top \in \bbR^n$.
  \begin{subequations}
    \begin{align}
      \label{eq:subgaussianVector}
      \bbE\left[e^{\lambda a^\top \bx}\right] &\le e^{4\lambda^2  b^2\|a\|_2^2}
      \\
     \nonumber 
  \bbP\left(
    \sum_{i=1}^na_i(\bx_i^2-\bbE[\bx_i^2]) > t
      \right) &\le \\
          \label{eq:vectorVarConcentration}
                                              &
               \hspace{-60pt}                                         
                                                \exp\left(-\frac{1}{64} \min\left\{\frac{t^2}{b^4\|a\|_2^2},\frac{t}{b^2\|a\|_{\infty}}\right\}\right)
    \end{align}
    \end{subequations}
\end{lemma}

\begin{IEEEproof}
  To prove (\ref{eq:subgaussianVector}), we use independence and (\ref{eq:orliczMGF}):
  \begin{align*}
    \bbE\left[e^{\lambda a^\top \bx}\right] &=\prod_{i=1}^n \bbE\left[e^{\lambda a_i \bx_i}\right]
                                              \overset{~(\ref{eq:orliczMGF})}{\le}e^{4\lambda^2 b^2 \sum_{i=1}^n a_i^2} 
  \end{align*}

  Now we prove (\ref{eq:vectorVarConcentration}).
  Without loss of generality, assume that $a_i\ne 0$, since the terms with $a_i=0$ can be dropped from the sum. Inequality (\ref{eq:squareSubexp}) shows that $\bx_i^2$ are all $(4\sqrt{2} b^2,4b^2)$-subexponential. It follows that $a_i\bx_i^2$ are all $(4\sqrt{2} b^2 |a_i|,4b^2|a_i|)$-subexponential. Direct calculation using independence shows that if $|\lambda| \le \frac{1}{4b^2 \|a\|_{\infty}}$, then
    $$
    \bbE\left[
      \exp\left(\lambda
        \sum_{i=1}^n a_i\bx_i^2
        \right)
      \right] \le \exp\left(
        (4b^2 \|a\|_2)^2 \lambda^2
      \right).
      $$
      Thus $\sum_{i=1}^na_i\bx_i^2$ is $(4\sqrt{2}b^2 \|a\|_2, 4b^2 \|a\|_{\infty})$-subexponential.

      Inequality (\ref{eq:vectorVarConcentration}) follows from (\ref{eq:subExpTail}) after noting that
      $$
      \min\left\{
        \frac{t^2}{32 b^4 \|a\|_2^2},\frac{t}{4b^2 \|a\|_{\infty}}
        \right\}\hspace{-2pt} \ge \hspace{-2pt} \frac{1}{32} \min\left\{
        \frac{t^2}{b^4 \|a\|_2^2},\frac{t}{b^2 \|a\|_{\infty}}
        \right\}.
      $$
    \end{IEEEproof}

    The following is the Hanson-Wright inequality stated with an explicit constant. 
    \begin{theorem}
      \label{thm:HW}
      Let $A\in\bbC^{n\times n}$ and assume that either $A\in\bbR^{n\times n}$ or $A$  is Hermitian.
      Let $\bx_i$ independent zero-mean scalar-valued random variables with $\|\bx_i\|_{\psi_2}\le b$ for $i=1,\ldots,n$. Let $\bx=\begin{bmatrix}\bx_1 & \cdots &\bx_n\end{bmatrix}^\top$. For all $t\ge 0$,  
      \begin{multline}
        \label{eq:HW}
        \bbP\left(\bx^\top A\bx - \bbE\left[\bx^\top A\bx\right] > \epsilon\right) \le \\
        2\exp\left(-\frac{1}{2048}\min\left\{\frac{\epsilon^2}{b^4\|A\|_F^2},\frac{\epsilon}{b^2\|A\|_2}\right\}\right)
      \end{multline}
    \end{theorem}

    \begin{IEEEproof}
      We sketch a variation of the proof of the Hanson-Wright inequality from \cite{vershynin2018high,rudelson2013hanson}, and make the associated constants explicit.

      Similar to the proof of Lemma~\ref{lem:gaussianHS}, let $B=\frac{1}{2}(A+A^\top)$ so that $B$ is a real symmetric matrix with $\bx^\top A \bx=\bx^\top B\bx$, $\|B\|_2\le \|A\|_2$, and $\|B\|_F\le \|A\|_F$.  
      
      First the probability is bounded in terms of the diagonal and off-diagonal terms:
      \begin{subequations}
        \begin{align}
          \nonumber
          \MoveEqLeft
          \bbP\left(\bx^\top A\bx - \bbE\left[\bx^\top A\bx\right] > \epsilon\right)\le
          \\
          &
            \bbP\left(\sum_{i=1}^n B_{ii} (\bx_i^2-\bbE[\bx_i^2]) > \epsilon/2\right)  + \\
          \label{eq:offdiagonalProb}
          &\bbP\left(\sum_{i\ne j} B_{ij}\bx_i\bx_j > \epsilon/2\right).
    \end{align}
    \end{subequations}

    If $a = \begin{bmatrix}B_{11} & \cdots  & B_{nn}\end{bmatrix}$, we have that $\|a\|_2 \le \|B\|_F\le\|A\|_F$ and $\|a\|_{\infty}\le\|B\|_2\le \|A\|_2$. So (\ref{eq:vectorVarConcentration}) implies that 
    \begin{multline*}
    \bbP\left(\sum_{i=1}^n B_{ii} (\bx_i^2-\bbE[\bx_i^2]) > \epsilon/2\right)\le\\
    \exp\left(-\frac{1}{256} 
\min\left\{
        \frac{\epsilon^2}{b^4 \|A\|_F^2},\frac{\epsilon}{b^2 \|A\|_{2}}\right\}\right).
  \end{multline*}

  We will show that the off-diagonal term, $\sum_{i\ne j}B_{ij}\bx_i\bx_j$, is $(16b^2 \|B\|_F,16b^2\|B\|_2)$-sub-exponential, and thus $(16b^2\|A\|_F,16b^2\|A\|_2)$-sub-exponential. Then (\ref{eq:subExpTail}) imples:
  \begin{align*}
    \MoveEqLeft
    \bbP\left(\sum_{i\ne j} B_{ij}\bx_i\bx_j > \epsilon/2\right)
    \\
    &\le
                                                           \exp\left(
                                                           -\frac{1}{2}\min
                                                           \left\{
                                                            \frac{
      \left(\epsilon/2\right)^2}{256 b^4 \|A\|_F^2},
      \frac{(\epsilon/2)}{16 b^2 \|A\|_2}
                                                            \right\}
      \right) \\
    &\le \exp\left(-\frac{1}{2048}\min\left\{\frac{\epsilon^2}{b^4\|A\|_F^2},\frac{\epsilon}{b^2\|A\|_2} \right\} \right).
  \end{align*}

  So we have
  \begin{multline*}
  \bbP\left(\bx^\top A\bx - \bbE\left[\bx^\top A\bx\right] > \epsilon\right) \le \\
  2\exp\left(-\frac{1}{2048}\min\left\{\frac{\epsilon^2}{b^4\|A\|_F^2},\frac{\epsilon}{b^2\|A\|_2} \right\} \right).
\end{multline*}

What remains is to prove that $\sum_{i\ne j} B_{ij}\bx_i \bx_j$ is sub-exponential. 
 Let $\bdelta_i$ be IID Bernoulli random variables with $\bbP(\bdelta_i=1)=\frac{1}{2}$. Let $\bdelta = \begin{bmatrix}\bdelta_1  & \cdots & \bdelta_n\end{bmatrix}^\top$ and set $B_{\bdelta} = \diag(\bdelta)B\diag(1_{n\times 1}-\bdelta)$.  Then
  $$
  \sum_{i\ne j}B_{ij}\bx_i\bx_j = 4\bbE_{\bdelta}[\bx^\top B_{\bdelta}\bx],
  $$
  where $\bbE_{\bdelta}$ corresponds to averaging over $\bdelta$ while keeping $\bx$ fixed.

  Let $\bx'$ be identically distributed to $\bx$ and independent of $\bx$.
  Then $\bx^\top B_{\bdelta}\bx$ is identically distributed to $\bx^\top B_{\bdelta}\bx'$. So, we have
  \begin{align*}
    \MoveEqLeft
    \bbE\left[
    \exp\left(\lambda   \sum_{i\ne j}B_{ij}\bx_i\bx_j\right)
    \right]
    =
     \bbE\left[
    \exp\left(4\lambda  \bbE_{\bdelta}\left[\bx^\top B_{\bdelta} \bx\right]\right)
      \right] \\
    &\overset{\textrm{Jensen}}{\le} \bbE\left[
      \exp\left(
      4\lambda \bx^\top B_{\bdelta} \bx
      \right)
      \right] \\
    &=\bbE\left[
      \exp\left(4\lambda \bx^\top B_{\bdelta}\bx'\right)
      \right] \\
    &\overset{~(\ref{eq:subgaussianVector})}{\le}
      \bbE\left[
      \exp\left(16 b^2 \lambda^2 \|B_{\bdelta}\bx'\|_2^2\right)
      \right]
  \end{align*}

  Let $\bg\in \bbR^n$ be a mean-zero Gaussian vectors with identity covariance independent of $\bx'$, and $\bdelta$.  
  \begin{align*}
       \bbE\left[
      \exp\left(16 b^2 \lambda^2 \|B_{\bdelta}\bx'\|_2^2\right)
      \right]
    &\overset{\mu:=\sqrt{32}b\lambda}{=}
      \bbE\left[
      \exp\left(
      \frac{1}{2}\mu^2 \|B_{\bdelta}\bx'\|_2^2
      \right)
      \right] \\
    &= \bbE\left[
      \exp\left(\mu \bg^\top B_{\bdelta}\bx'
      \right)\right] \\
    &\overset{~(\ref{eq:subgaussianVector})}{\le}
      \bbE\left[\exp\left(
      4\mu^2 b^2 \|B_{\bdelta}^\top \bg\|_2^2
      \right)\right] \\
    &=       \bbE\left[\exp\left(
      128\lambda^2 b^4 \|B_{\bdelta}^\top \bg\|_2^2
      \right)\right] .
  \end{align*}
  Now let $\bw = \bV\bg$ where $\bV$ is an orthogonal matrix such that $\bV B_{\bdelta} B_{\bdelta}^\top \bV^\top=\diag(\bs_1^2,\ldots,\bs_n^2)$, where $\bs_1,\ldots,\bs_n$ are the singular values of $B_{\bdelta}$. Then $\bw$ is also normally distributed with mean $0$ and covariance $I$. Let $\bbE_{\bw}$ denote expectation with respect to $\bw$ while holding the other variables fixed. 

  Now if $|\lambda| \le \frac{1}{16 b^2 \|B\|_2}$ we have $128\lambda^2 b^4 \bs_i^2 \le \frac{1}{2}$, since $\bs_i^2\le \|B_{\bdelta}\|_2^2 \le \|B\|_2^2$. In this case we have
  \begin{align*}
    \MoveEqLeft
\bbE\left[\exp\left(
      64\lambda^2 b^4 \|B_{\bdelta}^\top \bg\|_2^2
    \right)\right] \\
    &= \bbE\left[\prod_{i=1}^n \bbE_{\bw}\left[
                    \exp\left(128\lambda^2 b^4 \bs_i^2 \bw_i^2\right)
                    \right]\right] \\
                  &=\bbE\left[\prod_{i=1}^n\frac{1}{\sqrt{1-128\lambda^2 b^4\bs_i^2}}\right] \\
                  &\le \bbE\left[\prod_{i=1}^n\exp\left( 128\lambda^2 b^4\bs_i^2 \right) \right] \\
    &\overset{\|B_{\bdelta}\|_F\le \|B\|_F}{\le} e^{128\lambda^2 b^4 \|B\|_F^2}.
  \end{align*}
  The first inequality follows because
  $
  \frac{1}{\sqrt{1-x}} \le e^x
  $
  for all $x\in [0,1/2]$. It follows that the off-diagonal term is $(16 b^2 \|B\|_F,16 b^2 \|B\|_2)$-sub-exponential. 
   \end{IEEEproof}

\section{Biography Section}

\begin{IEEEbiography}[{\includegraphics[width=.8in,height=1in,clip,keepaspectratio]{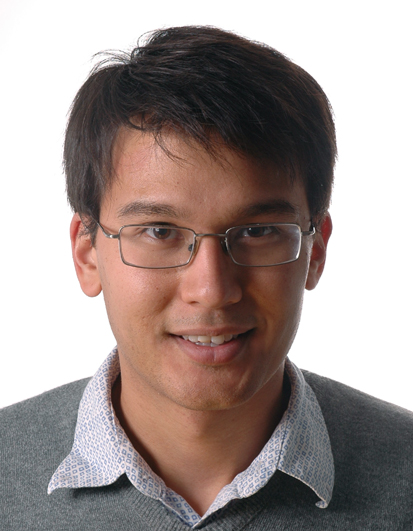}}] {Andrew Lamperski}
(S'05--M'11) received the B.S. degree in biomedical engineering and
mathematics in 2004 from the Johns Hopkins University, Baltimore, MD,
and the Ph.D. degree in control and dynamical systems in 2011 from the
California Institute of Technology, Pasadena. He held postdoctoral
positions in control and dynamical systems at the California Institute
of Technology from 2011--2012 and in mechanical engineering at The
Johns Hopkins University in 2012. From 2012--2014,
did
postdoctoral work in the Department of Engineering, University of
Cambridge, on a scholarship from the Whitaker International
Program. In 2014, he joined the Department of Electrical and Computer
Engineering, University of Minnesota, where he is currently an Associate Professor.
\end{IEEEbiography}\vspace{-40px}

\vfill

\end{document}